\documentclass{amsart}
\usepackage{amsmath}
\usepackage{amsfonts}
\usepackage{graphics}
\usepackage{epsfig}
\usepackage{amssymb}
\usepackage{amscd}
\usepackage[all]{xy}
\usepackage{latexsym}
\usepackage{graphicx}
\usepackage{multirow}
\usepackage{geometry}

\theoremstyle{plain}
\newtheorem{thm}{Theorem}[section]
\newtheorem{lem}[thm]{Lemma}

\theoremstyle{definition}
\newtheorem{defn}[thm]{Definition}

\theoremstyle{remark}

\numberwithin{equation}{section} \numberwithin{figure}{section}

\renewcommand*{\to}{\rightarrow}
\renewcommand*{\bar}[1]{\overline{#1}}

\newcommand{\mb}[1]{\mathbb{#1}} 

\newcommand{\mc}[1]{\mathcal{#1}}
\newcommand{\mk}[1]{\mathfrak{#1}}

\newcommand{\mf}[1]{\mathbf{#1}}

\usepackage{url}
\usepackage{pgfplots}
\pgfplotsset{compat=1.18}
\usetikzlibrary{calc} 
\usepackage{mathtools}
\usepackage{tikz-cd}

\title{An Explicit Model for Conic Laplacians on \(\mb P^1\)}

\begin{document}\fontsize{12}{20pt}\selectfont
\author{Jia-Ming (Frank) Liou}
\address{Department of Mathematics\\
National Cheng Kung University\\
No.1, University Road, Tainan City 701, Taiwan\\ fjmliou@mail.ncku.edu.tw}

\begin{abstract}
We study an explicit model of a conic Laplacian on \(\mb P^1\) with cone angle \(2\pi n\) at \(0\) and \(\infty\). Using Fourier decomposition, we reduce the eigenvalue equation to a family of associated Legendre equations 
and describe the corresponding asymptotic boundary data. The classical Legendre connection formulas induce an explicit gluing map between the boundary data at the two conic points. We compute the Friedrichs spectrum 
and eigenfunctions and show that the Weyl function, and hence the associated \(S\)-matrix, can be recovered explicitly from this boundary connection map.

\end{abstract}
\maketitle


\section{Introduction}

The purpose of this paper is to develop in detail an explicit model of
a conic Laplacian on the complex projective line. Let \(n\geq2\), and
equip \(\mb P^1\) with the metric
\[
    ds_n^2
    =
    \frac{4n^2|x|^{2n-2}}{(1+|x|^{2n})^2}|dx|^2.
\]
This metric has cone angle \(2\pi n\) at both \(0\) and \(\infty\) and
is the pullback of the standard round metric under the branched covering
\(x\mapsto x^n\). It therefore provides a natural model in which the
boundary and spectral structures of a conic Laplacian can be described
completely and explicitly. For general background on metrics and
Laplacians with conic singularities, see, for example,
\cite{Troyanov,Cheeger}.

We begin by decomposing the eigenvalue equation into Fourier modes.
After a suitable change of variables, each radial equation becomes an
associated Legendre equation. This allows us to describe explicitly the
local solutions, their asymptotic behavior, and their
square-integrability properties. We use the standard asymptotic formulas
for Legendre and associated Legendre functions collected in
\cite[Chapter~14]{DLMF}. Among the square-integrable local modes, only
finitely many contribute nontrivial boundary data to the quotient of
the maximal domain by the minimal domain. Their leading asymptotic
coefficients determine the boundary data spaces at the two conic points.
The general construction of these boundary spaces for conic Laplacians
was established in \cite[Section~2]{LiouKrein}.

The coordinate transformation \(y=x^{-1}\) reverses the angular modes
and relates the local solutions near \(0\) and \(\infty\). Together with
the classical connection formulas for associated Legendre functions
\cite[\S14.9]{DLMF}, it induces an explicit boundary connection map
relating the asymptotic data of the same formal solution at the two
conic points. In this model, the Legendre connection formulas thus
acquire a concrete geometric meaning as gluing laws for conic boundary
data.

The Green form induces a nondegenerate skew-Hermitian form on the
boundary data space, and the self-adjoint extensions of the minimal
Laplacian are parametrized by its Lagrangian subspaces. This
boundary-symplectic description was developed for conic Laplacians in
\cite{LiouKrein}; see also \cite{GilMendoza,GilKrainerMendoza} for the
general theory of closed extensions of elliptic cone operators. The
Lagrangian subspace determined by the regular asymptotic terms defines
the Friedrichs extension. We compute its spectrum and eigenfunctions
explicitly, prove the completeness of the resulting eigenfunction
system, and determine the corresponding \(L^2\)-norms.

For a spectral parameter in the resolvent set of a self-adjoint
extension, we use the resolvent to construct the unique formal solution
of \((\Delta-\lambda)u=0\) with prescribed boundary data in a chosen
complementary Lagrangian subspace. This construction leads naturally to
the Weyl function and, after choosing symplectically dual bases, to the
associated \(S\)-matrix. For the Friedrichs extension, both are computed
explicitly. The \(S\)-matrix decomposes into \(2\times2\) blocks
corresponding to the pairs of Fourier modes coupled by the coordinate
transformation \(y=x^{-1}\). For a related use of Krein's formula and
the \(S\)-matrix on Euclidean surfaces with conical singularities, see
\cite{HillairetKokotov}.

A central result of the paper is that the Weyl function can be recovered
directly from the boundary connection map. Consequently, the same
Legendre connection data that relate local solutions near \(0\) and
\(\infty\) also determine the global relation between their singular and
regular boundary data, as encoded by the Weyl function and the
\(S\)-matrix.

This paper may be viewed as a companion to our previous work on Krein's
formula for conic Laplacians \cite{LiouKrein}. In that work, the
boundary symplectic space, self-adjoint extensions, Weyl functions, and
related spectral constructions were developed in a general setting.
Here we realize these structures concretely for the projective model on
\(\mb P^1\). The present approach provides a complementary viewpoint:
rather than beginning with the abstract formalism, we derive the
boundary and spectral objects directly from the Fourier decomposition,
the associated Legendre equations, their local asymptotics, and their
connection formulas.

The paper has therefore been written to be as self-contained and
concrete as possible. We introduce only the general notions needed for
this model and carry out the principal constructions explicitly. Our
aim is not merely to record the resulting formulas, but to explain how
the boundary symplectic space, Lagrangian boundary conditions, the Weyl
function, and the \(S\)-matrix emerge naturally from the underlying
differential equations and from the gluing of their asymptotic data at
the two conic points. In this way, the projective model provides an
accessible and concrete entry point to the boundary and spectral theory
of conic Laplacians and makes the corresponding abstract constructions
more transparent.

The paper is organized as follows. In Section~2, we solve the local mode
equations, identify the critical boundary data, and construct the
boundary connection map. In Section~3, we describe the self-adjoint
extensions and compute the spectrum and eigenfunctions of the
Friedrichs extension. In Section~4, we construct formal solutions for
spectral parameters in the resolvent set, compute the Weyl function and
the \(S\)-matrix, and express them in terms of the boundary connection
map.

\section{Local Solutions, Boundary Data, and Gluing}

Let \(\mb P^1=(\mb C^2\setminus\{0\})/\mb C^*\) be the complex projective
line, with homogeneous coordinates \((a_0:a_1)\). We denote the points
\((0:1)\) and \((1:0)\) by \(0\) and \(\infty\), respectively. Let
\(U_0=\{(a_0:a_1)\in\mb P^1:a_1\neq 0\}\) and
\(U_\infty=\{(a_0:a_1)\in\mb P^1:a_0\neq 0\}\).

On \(U_0\), we use the affine coordinate \(x:U_0\to\mb C\) defined by
\(x(a_0:a_1)=a_0/a_1\), and on \(U_\infty\), we use the affine coordinate
\(y:U_\infty\to\mb C\) defined by \(y(a_0:a_1)=a_1/a_0\). On
\(U_0\cap U_\infty\), these coordinates satisfy \(xy=1\). Equivalently, the
transition function is \(y\circ x^{-1}(z)=z^{-1}\) for \(z\in\mb C^*\).
Thus \(\{(U_0,x),(U_\infty,y)\}\) is the standard complex atlas on
\(\mb P^1\).

Let \(n\geq 1\) be an integer. We define a Hermitian tensor \(ds_n^2\) on
\(\mb P^1\) by prescribing its coordinate expressions. In the affine
coordinate \(x\) on \(U_0\), set
\[
    ds_n^2
    =
    \frac{4n^2|x|^{2n-2}}{(1+|x|^{2n})^2}|dx|^2,
\]
and in the affine coordinate \(y\) on \(U_\infty\), set
\[
    ds_n^2
    =
    \frac{4n^2|y|^{2n-2}}{(1+|y|^{2n})^2}|dy|^2.
\]
Since \(y=x^{-1}\) and \(dy=-x^{-2}\,dx\) on \(U_0\cap U_\infty\), we have
\[
    \frac{4n^2|y|^{2n-2}}{(1+|y|^{2n})^2}|dy|^2
    =
    \frac{4n^2|x|^{2n-2}}{(1+|x|^{2n})^2}|dx|^2.
\]
Hence the two local expressions agree on \(U_0\cap U_\infty\) and determine
a global Hermitian tensor \(ds_n^2\) on \(\mb P^1\). For \(n=1\), this is the standard round metric, which we denote by
\(ds_{\mathrm{rd}}^2\). For \(n\geq 2\), the tensor is positive definite on
\(\mb P^1\setminus\{0,\infty\}\) and vanishes at \(0\) and \(\infty\). On
the punctured surface, it defines a spherical conic metric with cone angle
\(2\pi n\) at each of the two conic points. From now on, we assume \(n\geq 2\), set \(X=\mb P^1\) and
\(S=\{0,\infty\}\), and write \(X'=X\setminus S\).

We recall our conventions for derivatives in the affine coordinate \(x\).
Since \(U_0'=U_0\setminus\{0\}=X'\), the restriction
\(x|_{X'}:X'\to\mb C^*\) is a coordinate diffeomorphism. By a slight abuse
of notation, we continue to denote this restriction by \(x\), and its inverse
by \(x^{-1}\). For \(u\in C^\infty(X')\), we denote its representation in
the \(x\)-coordinate by
\[
    \widetilde u
    =
    u\circ x^{-1}
    \in C^\infty(\mb C^*).
\]
If \(z\) denotes the standard complex coordinate on \(\mb C\), then, for
\(p\in X'\), we define
\[
    \frac{\partial u}{\partial x}(p)
    =
    \frac{\partial\widetilde u}{\partial z}\bigl(x(p)\bigr),
    \qquad
    \frac{\partial u}{\partial\bar x}(p)
    =
    \frac{\partial\widetilde u}{\partial\bar z}\bigl(x(p)\bigr).
\]
We also use polar coordinates on \(\mb C^*\). Let
\[
    P:(0,\infty)\times S^1\longrightarrow\mb C^*,
    \qquad
    P(r,\theta)=re^{i\theta}.
\]
By a slight abuse of notation, we write
\[
    u(r,\theta)
    :=
    (\widetilde u\circ P)(r,\theta)
    =
    \widetilde u(re^{i\theta}).
\]
The corresponding radial and angular derivatives are defined by
\[
    \frac{\partial u}{\partial r}
    :=
    \frac{\partial}{\partial r}(\widetilde u\circ P),
    \qquad
    \frac{\partial u}{\partial\theta}
    :=
    \frac{\partial}{\partial\theta}(\widetilde u\circ P).
\]
Equivalently, in the coordinate \(x=re^{i\theta}\),
\[
    \frac{\partial}{\partial r}
    =
    e^{i\theta}\frac{\partial}{\partial x}
    +
    e^{-i\theta}\frac{\partial}{\partial\bar x},
    \qquad
    \frac{\partial}{\partial\theta}
    =
    ir e^{i\theta}\frac{\partial}{\partial x}
    -
    ir e^{-i\theta}\frac{\partial}{\partial\bar x}.
\]
Let \(dA\) be the area form associated with \(ds_n^2\). In the affine
coordinate \(x\), we have
\[
    (x^{-1})^*dA
    =
    \frac{4n^2|z|^{2n-2}}{(1+|z|^{2n})^2}
    \frac{i}{2}\,dz\wedge d\bar z.
\]
Identifying \(z\) with the affine coordinate \(x\), we write this more
simply as
\[
    dA
    =
    \frac{4n^2|x|^{2n-2}}{(1+|x|^{2n})^2}
    \frac{i}{2}\,dx\wedge d\bar x.
\]
Thus, in polar coordinates \(x=re^{i\theta}\),
\[
    P^*(x^{-1})^*dA
    =
    \frac{4n^2r^{2n-1}}{(1+r^{2n})^2}\,dr\wedge d\theta.
\]
We use the same conventions in the affine coordinate \(y\) near
\(\infty\). Since \(U_\infty\setminus\{\infty\}=X'\), the restriction
\(y|_{X'}:X'\to\mb C^*\) is also a coordinate diffeomorphism. By the same
abuse of notation, we continue to denote this restriction by \(y\), and its
inverse by \(y^{-1}\). For \(u\in C^\infty(X')\), we denote its
representation in the \(y\)-coordinate by
\[
    \widetilde v
    =
    u\circ y^{-1}
    \in C^\infty(\mb C^*).
\]
In polar coordinates \(y=\rho e^{i\eta}\), we write
\[
    v(\rho,\eta)
    :=
    (\widetilde v\circ P)(\rho,\eta)
    =
    \widetilde v(\rho e^{i\eta}).
\]
The corresponding derivative identities and area-form formulas are obtained
from the preceding ones by replacing
\((x,r,\theta,\widetilde u,u)\) with
\((y,\rho,\eta,\widetilde v,v)\). In particular,
\[
    dA
    =
    \frac{4n^2|y|^{2n-2}}{(1+|y|^{2n})^2}
    \frac{i}{2}\,dy\wedge d\bar y
\]
in the affine coordinate \(y\), and
\[
    P^*(y^{-1})^*dA
    =
    \frac{4n^2\rho^{2n-1}}{(1+\rho^{2n})^2}\,
    d\rho\wedge d\eta.
\]
Let \(L^2(X',dA)\) be the space of square-integrable complex-valued
measurable functions on \(X'\), where functions that agree almost
everywhere are identified. It is a separable complex Hilbert space with
Hermitian inner product
\[
    \langle u,v\rangle_{L^2(X',dA)}
    =
    \int_{X'}u\bar v\,dA.
\]
We use the convention that the inner product is linear in the first
variable. With the sign convention that the Laplacian is nonnegative, let
\[
    \Delta_c:C_c^\infty(X')
    \subset L^2(X',dA)
    \longrightarrow L^2(X',dA)
\]
denote the Laplace--Beltrami operator associated with \(ds_n^2\), initially
defined on compactly supported smooth functions on \(X'\). In the affine
coordinate \(x\), it is given by
\[
    \Delta_c
    =
    -
    \frac{(1+|x|^{2n})^2}{n^2|x|^{2n-2}}
    \frac{\partial^2}{\partial x\,\partial\bar x}.
\]
Let \(u\in L^2(X',dA)\). Suppose that there exists
\(F\in L^2(X',dA)\) such that
\[
    \langle u,\Delta_c\varphi\rangle_{L^2(X',dA)}
    =
    \langle F,\varphi\rangle_{L^2(X',dA)}
\]
for every \(\varphi\in C_c^\infty(X')\). Since
\(C_c^\infty(X')\) is dense in \(L^2(X',dA)\), such an \(F\), if it
exists, is unique. We denote it by \(\Delta u\) and say that
\(\Delta u=F\) in the distributional sense on \(X'\).

\begin{defn}
The maximal domain of the Laplacian is
\[
    \mc D_{\max}(\Delta)
    =
    \bigl\{
        u\in L^2(X',dA):
        \Delta u\in L^2(X',dA)
        \text{ in the distributional sense}
    \bigr\}.
\]
\end{defn}
We equip \(\mc D_{\max}(\Delta)\) with the graph norm
\[
    \|u\|_{\Delta}
    =
    \left(
        \|u\|_{L^2(X',dA)}^2
        +
        \|\Delta u\|_{L^2(X',dA)}^2
    \right)^{1/2}.
\]
The maximal Laplacian
\[
    \Delta:\mc D_{\max}(\Delta)\longrightarrow L^2(X',dA)
\]
is the adjoint of the densely defined operator \(\Delta_c\), and is therefore
closed. Consequently, \(\mc D_{\max}(\Delta)\) is complete with respect to
the graph norm.

Moreover, \(C_c^\infty(X')\subset\mc D_{\max}(\Delta)\), and for every
\(\varphi\in C_c^\infty(X')\), the distributional Laplacian agrees with the
classical Laplace--Beltrami operator:
\[
    \Delta\varphi=\Delta_c\varphi.
\]

\begin{defn}
The minimal domain of the Laplacian is
\[
    \mc D_{\min}(\Delta)
    =
    \overline{C_c^\infty(X')}^{\,\|\cdot\|_\Delta}
    \subset\mc D_{\max}(\Delta),
\]
where the closure is taken with respect to the graph norm. Equivalently,
the operator
\[
    \Delta:\mc D_{\min}(\Delta)\longrightarrow L^2(X',dA)
\]
is the closure of \(\Delta_c\).
\end{defn}

\begin{defn}\label{def:formal-solution}
Let \(\lambda\in\mb C\). An element
\(u_\lambda\in\mc D_{\max}(\Delta)\) is called a formal solution of
\begin{equation}\label{eq:model}
    (\Delta-\lambda)u_\lambda=0
\end{equation}
if the equation holds on \(X'\) in the sense of distributions.
\end{defn}

For \(\lambda\in\mb C\), we first solve the differential equation
\eqref{eq:model} in the affine coordinate \(x\) on \(X'\), without initially
imposing square-integrability or the full
\(\mc D_{\max}(\Delta)\)-condition. We then determine which of these
solutions belong to \(L^2(X',dA)\) and which Fourier modes contribute
nontrivial asymptotic boundary data.

For \(u\in L^2(X',dA)\), define its Fourier coefficients by
\[
    u_k(r)
    =
    \frac{1}{2\pi}
    \int_{0}^{2\pi}u(r,\theta)e^{-ik\theta}\,d\theta,
    \qquad k\in\mb Z,
\]
for almost every \(r>0\). Then
\[
    u(r,\theta)
    =
    \sum_{k\in\mb Z}u_k(r)e^{ik\theta},
\]
where the series converges in \(L^2(S^1)\) for almost every \(r\), and hence
also in the sense of distributions on \((0,\infty)\times S^1\). The
distributional Laplacian of \(u\) is given by
\[
    \Delta u
    =
    -\frac{(1+r^{2n})^2}{4n^2r^{2n}}
    \sum_{k\in\mb Z}(L_k u_k)(r)e^{ik\theta},
\]
where
\[
    L_k v=r^2v''+rv'-k^2v,
\]
with the derivatives interpreted in the distributional sense.

Let \(\lambda\in\mb C\). Since the operator \(\Delta-\lambda\) is elliptic
on \(X'\), every distributional solution of
\[
    (\Delta-\lambda)u_\lambda=0
\]
is smooth on \(X'\). Thus, at this stage, we consider
\(u_\lambda\in C^\infty(X')\) satisfying this equation, without imposing
square-integrability at \(0\) or \(\infty\). Write
\[
    u_\lambda(r,\theta)
    =
    \sum_{k\in\mb Z}u_k(r;\lambda)e^{ik\theta}.
\]
Then, for every \(k\in\mb Z\), the Fourier coefficient
\(u_k(r;\lambda)\) satisfies
\[
    (L_k u_k)(r;\lambda)
    +
    \frac{4n^2\lambda r^{2n}}{(1+r^{2n})^2}
    u_k(r;\lambda)
    =
    0
\]
on \((0,\infty)\).

Set \(R=r^n\) and
\(U_k(R;\lambda)=u_k(R^{1/n};\lambda)\). Since
\(r\partial_r=nR\partial_R\), the preceding equation becomes
\[
    R^2\frac{d^2U_k}{dR^2}
    +
    R\frac{dU_k}{dR}
    +
    \left(
        \frac{4\lambda R^2}{(1+R^2)^2}
        -
        \frac{k^2}{n^2}
    \right)U_k
    =
    0.
\]

Next, let \(\phi=2\arctan R\), so that
\(R=\tan(\phi/2)\) and \(0<\phi<\pi\), and set
\[
    v_k(\phi;\lambda)
    =
    U_k\bigl(\tan(\phi/2);\lambda\bigr).
\]
Since \(R\partial_R=\sin\phi\,\partial_\phi\), the equation takes the form
\[
    \frac{d^2v_k}{d\phi^2}
    +
    \cot\phi\,\frac{dv_k}{d\phi}
    +
    \left(
        \lambda
        -
        \frac{\alpha_k^2}{\sin^2\phi}
    \right)v_k
    =
    0,
\]
where \(\alpha_k=|k|/n\).

Let \(t=\cos\phi\), and set
\(V_k(t;\lambda)=v_k(\arccos t;\lambda)\) for \(-1<t<1\). Choose
\(\nu\in\mb C\) such that \(\lambda=\nu(\nu+1)\). Then the equation becomes
the associated Legendre equation
\begin{equation}\label{eq:associated-legendre-model}
    (1-t^2)\frac{d^2V_k}{dt^2}
    -
    2t\frac{dV_k}{dt}
    +
    \left(
        \nu(\nu+1)
        -
        \frac{\alpha_k^2}{1-t^2}
    \right)V_k
    =
    0.
\end{equation}
We now consider separately the zero mode \(k=0\) and the nonzero modes
\(k\neq 0\).

Suppose first that \(k=0\). Then the equation becomes
\[
    (1-t^2)\frac{d^2V_0}{dt^2}
    -
    2t\frac{dV_0}{dt}
    +
    \nu(\nu+1)V_0
    =
    0.
\]
Choose \(\nu=\nu(\lambda)\notin\{-1,-2,\ldots\}\) such that
\(\lambda=\nu(\nu+1)\). The general solution is
\[
    V_0(t;\lambda)
    =
    A_0(\lambda)P_\nu(t)
    +
    B_0(\lambda)Q_\nu(t),
\]
where \(P_\nu\) and \(Q_\nu\) denote the Ferrers functions of the first and
second kinds, respectively, on \((-1,1)\). Returning to the variable \(r\),
we obtain
\[
    u_0(r;\lambda)
    =
    A_0(\lambda)
    P_\nu\left(\frac{1-r^{2n}}{1+r^{2n}}\right)
    +
    B_0(\lambda)
    Q_\nu\left(\frac{1-r^{2n}}{1+r^{2n}}\right).
\]
As \(t\to1^{-}\), we have \(P_\nu(t)=1+O(1-t)\) and
\[
    Q_\nu(t)
    =
    \frac{1}{2}\log 2
    -
    \frac{1}{2}\log(1-t)
    -
    \gamma_E
    -
    \psi(\nu+1)
    +
    O\bigl((1-t)|\log(1-t)|\bigr),
\]
where \(\gamma_E\) is the Euler constant and
\(\psi=\Gamma'/\Gamma\) is the digamma function. Since
\[
    t=\frac{1-r^{2n}}{1+r^{2n}},
    \qquad
    1-t=\frac{2r^{2n}}{1+r^{2n}},
\]
it follows that, as \(r\to0^{+}\),
\[
    P_\nu\left(\frac{1-r^{2n}}{1+r^{2n}}\right)
    =
    1+O(r^{2n})
\]
and
\[
    Q_\nu\left(\frac{1-r^{2n}}{1+r^{2n}}\right)
    =
    -n\log r
    -
    \gamma_E
    -
    \psi(\nu+1)
    +
    O\bigl(r^{2n}|\log r|\bigr).
\]
Consequently,
\[
    u_0(r;\lambda)
    =
    A_0(\lambda)
    -
    B_0(\lambda)\bigl(\gamma_E+\psi(\nu+1)\bigr)
    -
    nB_0(\lambda)\log r
    +
    O\bigl(r^{2n}|\log r|\bigr).
\]
We next assume that \(k\neq0\) and \(\alpha_k\notin\mb Z\). Then
\(P_\nu^{-\alpha_k}\) and \(P_\nu^{\alpha_k}\) form a pair of linearly
independent solutions of the associated Legendre equation
\eqref{eq:associated-legendre-model}. Here and below, since
\(-1<t<1\), the notation \(P_\nu^\mu(t)\) refers to the Ferrers function
of the first kind. Returning to the variable \(r\), we obtain
\[
    u_k(r;\lambda)
    =
    A_k(\lambda)
    P_\nu^{-\alpha_k}
    \left(\frac{1-r^{2n}}{1+r^{2n}}\right)
    +
    B_k(\lambda)
    P_\nu^{\alpha_k}
    \left(\frac{1-r^{2n}}{1+r^{2n}}\right).
\]
Using the hypergeometric representation
\[
    P_\nu^\mu(t)
    =
    \frac{1}{\Gamma(1-\mu)}
    \left(\frac{1+t}{1-t}\right)^{\mu/2}
    {}_2F_1\left(
        -\nu,\nu+1;1-\mu;\frac{1-t}{2}
    \right),
\]
we obtain, as \(t\to1^{-}\),
\[
    P_\nu^\mu(t)
    =
    \frac{1}{\Gamma(1-\mu)}
    \left(\frac{1+t}{1-t}\right)^{\mu/2}
    \bigl(1+O(1-t)\bigr).
\]
For \(t=(1-r^{2n})/(1+r^{2n})\), we have
\[
    \frac{1+t}{1-t}=r^{-2n},
    \qquad
    \frac{1-t}{2}
    =
    \frac{r^{2n}}{1+r^{2n}}.
\]
It follows that, as \(r\to0^{+}\),
\[
    P_\nu^\mu
    \left(\frac{1-r^{2n}}{1+r^{2n}}\right)
    =
    \frac{1}{\Gamma(1-\mu)}
    r^{-n\mu}
    \bigl(1+O(r^{2n})\bigr).
\]
In particular,
\[
    P_\nu^{-\alpha_k}
    \left(\frac{1-r^{2n}}{1+r^{2n}}\right)
    =
    \frac{1}{\Gamma(1+\alpha_k)}
    r^{|k|}
    \bigl(1+O(r^{2n})\bigr),
\]
whereas
\[
    P_\nu^{\alpha_k}
    \left(\frac{1-r^{2n}}{1+r^{2n}}\right)
    =
    \frac{1}{\Gamma(1-\alpha_k)}
    r^{-|k|}
    \bigl(1+O(r^{2n})\bigr).
\]
Consequently,
\[
    u_k(r;\lambda)
    =
    \frac{A_k(\lambda)}{\Gamma(1+\alpha_k)}r^{|k|}
    +
    \frac{B_k(\lambda)}{\Gamma(1-\alpha_k)}r^{-|k|}
    +
    O(r^{2n-|k|})
\]
as \(r\to0^{+}\).

When \(\alpha_k\in\mb Z_{>0}\), the functions
\(P_\nu^{-\alpha_k}\) and \(P_\nu^{\alpha_k}\) are no longer linearly
independent. Set \(m=\alpha_k\). A convenient fundamental system, valid for
all \(\nu\), is given by \(P_\nu^{-m}\) and \(Q_\nu^m\). Thus the \(k\)-th
radial component has the form
\[
    u_k(r;\lambda)
    =
    A_k(\lambda)
    P_\nu^{-m}\left(\frac{1-r^{2n}}{1+r^{2n}}\right)
    +
    B_k(\lambda)
    Q_\nu^m\left(\frac{1-r^{2n}}{1+r^{2n}}\right).
\]
As \(t\to1^{-}\), the hypergeometric representation gives
\[
    P_\nu^{-m}(t)
    =
    \frac{1}{m!}
    \left(\frac{1-t}{1+t}\right)^{m/2}
    \bigl(1+O(1-t)\bigr),
\]
whereas
\[
    Q_\nu^m(t)
    =
    (-1)^m2^{m-1}(m-1)!
    (1-t^2)^{-m/2}
    \bigl(1+O((1-t)|\log(1-t)|)\bigr).
\]
Since
\[
    \frac{1-t}{1+t}=r^{2n},
    \qquad
    1-t^2=\frac{4r^{2n}}{(1+r^{2n})^2},
\]
we obtain, as \(r\to0^{+}\),
\[
    P_\nu^{-m}\left(\frac{1-r^{2n}}{1+r^{2n}}\right)
    =
    \frac{1}{m!}r^{nm}
    \bigl(1+O(r^{2n})\bigr)
\]
and
\[
    Q_\nu^m\left(\frac{1-r^{2n}}{1+r^{2n}}\right)
    =
    \frac{(-1)^m}{2}(m-1)!
    r^{-nm}
    \bigl(1+O(r^{2n}|\log r|)\bigr).
\]
Consequently, since \(nm=|k|\),
\[
    u_k(r;\lambda)
    =
    \frac{A_k(\lambda)}{m!}r^{|k|}
    +
    \frac{(-1)^m}{2}(m-1)!B_k(\lambda)r^{-|k|}
    +
    O\bigl(r^{2n-|k|}|\log r|\bigr)
\]
as \(r\to0^{+}\).

In both the non-integer-order and integer-order cases, the singular branch
has leading behavior \(r^{-|k|}\). Since
\(dA\sim C r^{2n-1}\,dr\,d\theta\) as \(r\to0^{+}\), this branch is
square-integrable near \(r=0\) if and only if
\[
    \int_0^\varepsilon
    r^{-2|k|}r^{2n-1}\,dr
    <
    \infty,
\]
or equivalently, if and only if \(|k|<n\). Hence the \(L^2\)-condition
forces
\[
    B_k(\lambda)=0
    \qquad\text{whenever } |k|\geq n.
\]
This is the first indication that the asymptotic boundary data are
finite-dimensional. We call the Fourier modes with \(|k|<n\) the critical
asymptotic modes. For \(0<|k|<n\), these are precisely the modes for which
both the regular branch \(r^{|k|}e^{ik\theta}\) and the singular branch
\(r^{-|k|}e^{ik\theta}\) are square-integrable near the conic point. We also
include the zero mode, since both \(1\) and \(\log r\) are square-integrable
with respect to \(dA\) near \(r=0\). Thus, in the critical modes, square-integrability alone does not eliminate
either of the two leading asymptotic coefficients. These coefficients will
subsequently be shown to determine the critical boundary data in
\(\mc D_{\max}(\Delta)/\mc D_{\min}(\Delta)\). By contrast, when
\(|k|\geq n\), the singular branch is not square-integrable and is therefore
excluded from the maximal domain.

We first finish the discussion of these complementary modes. For
\(|k|\geq n\), any square-integrable formal solution in the \(k\)-th mode
must be of the form
\[
    u_k(r;\lambda)
    =
    A_k(\lambda)
    P_\nu^{\mu_k}
    \left(\frac{1-r^{2n}}{1+r^{2n}}\right),
\]
where
\[
    \mu_k
    =
    \begin{cases}
        -\alpha_k, & \alpha_k\notin\mb Z,\\
        \alpha_k, & \alpha_k\in\mb Z_{>0}.
    \end{cases}
\]
After imposing the \(L^2\)-condition near \(r=0\), it remains to determine
when this branch is square-integrable at \(r=\infty\). Equivalently, we seek
the necessary and sufficient condition under which
\[
    P_\nu^{\mu_k}
    \left(\frac{1-r^{2n}}{1+r^{2n}}\right)e^{ik\theta}
    \in L^2(\mb C^*,dA).
\]

Let \(y=\rho e^{i\eta}=1/x\). Then \(\rho=r^{-1}\),
\(\eta=-\theta\), and
\[
    \frac{1-r^{2n}}{1+r^{2n}}
    =
    -
    \frac{1-\rho^{2n}}{1+\rho^{2n}}.
\]
We consider separately the cases in the definition of \(\mu_k\).

Suppose first that \(\alpha_k\notin\mb Z\), so that
\(\mu_k=-\alpha_k\). Using the hypergeometric representation of
\(P_\nu^{-\alpha_k}\) and the connection formula for the hypergeometric
function at \(1\), one finds that the coefficient of the singular branch
\(\rho^{-|k|}\) at \(\rho=0\) is
\[
    c_k(\nu)
    =
    \frac{\Gamma(\alpha_k)}
    {
        \Gamma(1+\alpha_k+\nu)
        \Gamma(\alpha_k-\nu)
    }.
\]
Since \(|k|\geq n\), the branch
\(\rho^{-|k|}e^{-ik\eta}\) is not square-integrable near \(\rho=0\).
Hence the solution is square-integrable at \(r=\infty\) if and only if
\(c_k(\nu)=0\). Since the gamma function has no zeros and has poles at the
nonpositive integers, this is equivalent to
\[
    \nu=\alpha_k+q
    \qquad\text{or}\qquad
    \nu=-\alpha_k-q-1
\]
for some \(q\in\mb Z_{\geq0}\).

Now suppose that \(\alpha_k=m\in\mb Z_{>0}\), so that \(\mu_k=m\). By the
connection formula,
\[
    P_\nu^m(-t)
    =
    \cos\pi(\nu+m)P_\nu^m(t)
    -
    \frac{2}{\pi}\sin\pi(\nu+m)Q_\nu^m(t).
\]
The \(Q_\nu^m\)-term has leading behavior of order
\((1-t)^{-m/2}\), which corresponds to the non-square-integrable branch
\(\rho^{-|k|}\) at \(\rho=0\). Thus, whenever
\(P_\nu^m\not\equiv0\), square-integrability at \(r=\infty\) requires
\(\sin\pi(\nu+m)=0\). Together with the nonvanishing condition for
\(P_\nu^m\), this gives
\[
    \nu=m+q
    \qquad\text{or}\qquad
    \nu=-m-q-1
\]
for some \(q\in\mb Z_{\geq0}\).

Consequently, in both cases,
\[
    \nu=\alpha_k+q
    \qquad\text{or}\qquad
    \nu=-\alpha_k-q-1,
    \qquad
    q\in\mb Z_{\geq0}.
\]
The two possibilities are related by the symmetry
\(\nu\mapsto-\nu-1\). Hence, up to this symmetry, we may write
\(\nu=\alpha_k+q\), and therefore
\[
    \lambda
    =
    (\alpha_k+q)(\alpha_k+q+1).
\]

We conclude the following.

\begin{lem}\label{lem:large-mode-regularity-condition}
Let \(k\in\mb Z\) satisfy \(|k|\geq n\), write
\(\lambda=\nu(\nu+1)\), and let
\[
    \mu_k
    =
    \begin{cases}
        -\alpha_k, & \alpha_k\notin\mb Z,\\
        \alpha_k, & \alpha_k\in\mb Z_{>0}.
    \end{cases}
\]
Define
\[
    \Psi_k(r,\theta;\lambda)
    =
    P_\nu^{\mu_k}
    \left(\frac{1-r^{2n}}{1+r^{2n}}\right)e^{ik\theta},
    \qquad
    x=re^{i\theta}.
\]
Then \(\Psi_k(\,\cdot\,;\lambda)\) is a nonzero square-integrable formal
solution on \(X'\) if and only if
\[
    \nu=\alpha_k+q
    \qquad\text{or}\qquad
    \nu=-\alpha_k-q-1
\]
for some \(q\in\mb Z_{\geq0}\). Equivalently, up to the symmetry
\(\nu\mapsto-\nu-1\), one has \(\nu=\alpha_k+q\). In this case,
\[
    \lambda
    =
    (\alpha_k+q)(\alpha_k+q+1).
\]
\end{lem}

Let \(\lambda\in\mb C\), and choose \(\nu\in\mb C\) such that
\(\lambda=\nu(\nu+1)\). Suppose that \(u_\lambda\) is a formal solution of
\eqref{eq:model}. In polar coordinates, write
\[
    u_\lambda(r,\theta)
    =
    \sum_{k\in\mb Z}u_k(r;\lambda)e^{ik\theta}.
\]
For the zero mode, write
\[
    u_0(r;\lambda)
    =
    A_0(\lambda)
    P_\nu\left(\frac{1-r^{2n}}{1+r^{2n}}\right)
    +
    B_0(\lambda)
    Q_\nu\left(\frac{1-r^{2n}}{1+r^{2n}}\right).
\]
For the critical nonzero modes \(0<|k|\leq n-1\), write
\[
    u_k(r;\lambda)
    =
    A_k(\lambda)
    P_\nu^{-\alpha_k}
    \left(\frac{1-r^{2n}}{1+r^{2n}}\right)
    +
    B_k(\lambda)
    P_\nu^{\alpha_k}
    \left(\frac{1-r^{2n}}{1+r^{2n}}\right),
\]
where \(\alpha_k=|k|/n\).

The asymptotic formulas obtained above show that the boundary data of
\(u_\lambda\) at \(0\) are encoded by the leading coefficients of its
critical Fourier modes as \(r\to0^+\). The zero mode contributes the model
terms \(1\) and \(\log|x|\). For each \(1\leq k\leq n-1\), the pair of
Fourier modes \(k\) and \(-k\) contributes the regular terms \(x^k\) and
\(\bar x^k\), together with the singular terms \(x^{-k}\) and
\(\bar x^{-k}\). More precisely,
\[
    r^ke^{ik\theta}=x^k,
    \qquad
    r^ke^{-ik\theta}=\bar x^k,
\]
whereas
\[
    r^{-k}e^{-ik\theta}=x^{-k},
    \qquad
    r^{-k}e^{ik\theta}=\bar x^{-k}.
\]
These model terms span a finite-dimensional space that records the possible
critical asymptotic behavior of \(u_\lambda\) near \(0\). The corresponding construction at \(\infty\) is obtained in the affine
coordinate \(y=\rho e^{i\eta}\), by examining the leading behavior as
\(\rho\to0^+\). These observations lead naturally to the definition of the
critical asymptotic spaces at the two conic points.

\begin{defn}
The critical asymptotic space at \(0\) is the finite-dimensional complex
vector space
\[
    V_0
    =
    \operatorname{span}_{\mb C}
    \bigl\{
        1,\log|x|,
        x^k,\bar x^k,x^{-k},\bar x^{-k}
        :
        1\leq k\leq n-1
    \bigr\},
\]
where the elements are regarded as model asymptotic terms on the punctured
affine chart \(U_0'=U_0\setminus\{0\}\). Similarly, the critical asymptotic
space at \(\infty\) is
\[
    V_\infty
    =
    \operatorname{span}_{\mb C}
    \bigl\{
        1,\log|y|,
        y^k,\bar y^k,y^{-k},\bar y^{-k}
        :
        1\leq k\leq n-1
    \bigr\},
\]
where the elements are regarded as model asymptotic terms on
\(U_\infty'=U_\infty\setminus\{\infty\}\). The critical asymptotic space
associated with the conic set \(S=\{0,\infty\}\) is
\[
    V_S=V_0\oplus V_\infty.
\]
\end{defn}

On \(U_0'\), write \(x=re^{i\theta}\), where \(r>0\) and
\(\theta\in S^1\). Define \(f_0=1/\sqrt{2\pi}\) and
\[
    f_0^\#
    =
    -
    \frac{1}{\sqrt{2\pi}}\log r.
\]
For \(j\in\mb Z\setminus\{0\}\) with \(1\leq |j|\leq n-1\), set
\[
    f_j
    =
    \frac{1}{\sqrt{4\pi|j|}}
    r^{|j|}e^{ij\theta},
    \qquad
    f_j^\#
    =
    \frac{1}{\sqrt{4\pi|j|}}
    r^{-|j|}e^{ij\theta}.
\]
Let
\[
    \mc F_0
    =
    \operatorname{span}_{\mb C}
    \{f_j:|j|\leq n-1\},
    \qquad
    \mc F_0^\#
    =
    \operatorname{span}_{\mb C}
    \{f_j^\#:|j|\leq n-1\}.
\]
Then
\[
    V_0=\mc F_0\oplus\mc F_0^\#.
\]

Similarly, on \(U_\infty'\), write \(y=\rho e^{i\eta}\), where
\(\rho>0\) and \(\eta\in S^1\). Define \(g_0=1/\sqrt{2\pi}\) and
\[
    g_0^\#
    =
    -
    \frac{1}{\sqrt{2\pi}}\log\rho.
\]
For \(j\in\mb Z\setminus\{0\}\) with \(1\leq |j|\leq n-1\), set
\[
    g_j
    =
    \frac{1}{\sqrt{4\pi|j|}}
    \rho^{|j|}e^{ij\eta},
    \qquad
    g_j^\#
    =
    \frac{1}{\sqrt{4\pi|j|}}
    \rho^{-|j|}e^{ij\eta}.
\]
Let
\[
    \mc F_\infty
    =
    \operatorname{span}_{\mb C}
    \{g_j:|j|\leq n-1\},
    \qquad
    \mc F_\infty^\#
    =
    \operatorname{span}_{\mb C}
    \{g_j^\#:|j|\leq n-1\}.
\]
Then
\[
    V_\infty=\mc F_\infty\oplus\mc F_\infty^\#.
\]

Let \(u_\lambda\) be a formal solution of \eqref{eq:model}. We define
\(\pi_0(u_\lambda)\in V_0\) by retaining the leading critical asymptotic
terms of \(u_\lambda\) as \(r\to0^+\). More precisely, write the radial
coefficients of the critical modes near \(0\) as
\[
    u_0(r;\lambda)
    =
    c_0(\lambda)
    +
    c_0^\#(\lambda)\log r
    +
    O(r^{2n}\log r)
\]
and, for \(0<|k|\leq n-1\),
\[
    u_k(r;\lambda)
    =
    c_k(\lambda)r^{|k|}
    +
    c_k^\#(\lambda)r^{-|k|}
    +
    O(r^{2n-|k|})
    \qquad
    \text{as } r\to0^+.
\]
We then set
\[
    \pi_0(u_\lambda)
    =
    c_0(\lambda)
    +
    c_0^\#(\lambda)\log|x|
    +
    \sum_{0<|k|\leq n-1}
    \left(
        c_k(\lambda)r^{|k|}e^{ik\theta}
        +
        c_k^\#(\lambda)r^{-|k|}e^{ik\theta}
    \right),
    \qquad
    x=re^{i\theta}.
\]
Equivalently,
\[
\begin{aligned}
    \pi_0(u_\lambda)
    &=
    c_0(\lambda)
    +
    c_0^\#(\lambda)\log|x|
    \\
    &\quad
    +
    \sum_{k=1}^{n-1}
    \left(
        c_k(\lambda)x^k
        +
        c_{-k}(\lambda)\bar x^k
        +
        c_{-k}^\#(\lambda)x^{-k}
        +
        c_k^\#(\lambda)\bar x^{-k}
    \right).
\end{aligned}
\]
By the asymptotic formulas obtained above,
\[
    c_0(\lambda)
    =
    A_0(\lambda)
    -
    B_0(\lambda)\bigl(\gamma_E+\psi(\nu+1)\bigr),
    \qquad
    c_0^\#(\lambda)
    =
    -nB_0(\lambda),
\]
and, for \(0<|k|\leq n-1\),
\[
    c_k(\lambda)
    =
    \frac{A_k(\lambda)}{\Gamma(1+\alpha_k)},
    \qquad
    c_k^\#(\lambda)
    =
    \frac{B_k(\lambda)}{\Gamma(1-\alpha_k)}.
\]
Since \(\alpha_{\pm k}=k/n\) for \(1\leq k\leq n-1\), it follows that
\[
\begin{aligned}
    \pi_0(u_\lambda)
    &=
    A_0(\lambda)
    -
    B_0(\lambda)\bigl(\gamma_E+\psi(\nu+1)\bigr)
    -
    nB_0(\lambda)\log|x|
    \\
    &\quad
    +
    \sum_{k=1}^{n-1}
    \left[
        \frac{A_k(\lambda)}{\Gamma(1+k/n)}x^k
        +
        \frac{A_{-k}(\lambda)}{\Gamma(1+k/n)}\bar x^k
    \right]
    \\
    &\quad
    +
    \sum_{k=1}^{n-1}
    \left[
        \frac{B_{-k}(\lambda)}{\Gamma(1-k/n)}x^{-k}
        +
        \frac{B_k(\lambda)}{\Gamma(1-k/n)}\bar x^{-k}
    \right].
\end{aligned}
\]

Equivalently, in terms of the bases \(f_j\) and \(f_j^\#\), we have
\begin{equation}\label{eq:pi0-critical-asymptotics}
\begin{aligned}
    \pi_0(u_\lambda)
    &=
    \sqrt{2\pi}
    \left[
        A_0(\lambda)
        -
        B_0(\lambda)\bigl(\gamma_E+\psi(\nu+1)\bigr)
    \right]f_0
    +
    n\sqrt{2\pi}\,B_0(\lambda)f_0^\#
    \\
    &\quad
    +
    \sum_{0<|k|\leq n-1}
    \sqrt{4\pi|k|}
    \left[
        \frac{A_k(\lambda)}{\Gamma(1+\alpha_k)}f_k
        +
        \frac{B_k(\lambda)}{\Gamma(1-\alpha_k)}f_k^\#
    \right].
\end{aligned}
\end{equation}

Similarly, using the coordinate \(y=\rho e^{i\eta}\) near \(\infty\), we
define \(\pi_\infty(u_\lambda)\in V_\infty\) by retaining the leading
critical asymptotic terms of \(v_\lambda(\rho,\eta)\) as
\(\rho\to0^+\). More precisely, suppose that the radial coefficients of the
critical modes have asymptotic expansions
\[
    v_0(\rho;\lambda)
    =
    d_0(\lambda)
    +
    d_0^\#(\lambda)\log\rho
    +
    O(\rho^{2n}\log\rho)
\]
and, for \(0<|k|\leq n-1\),
\[
    v_k(\rho;\lambda)
    =
    d_k(\lambda)\rho^{|k|}
    +
    d_k^\#(\lambda)\rho^{-|k|}
    +
    O(\rho^{2n-|k|})
    \qquad
    \text{as }\rho\to0^+.
\]
We then set
\[
    \pi_\infty(u_\lambda)
    =
    d_0(\lambda)
    +
    d_0^\#(\lambda)\log|y|
    +
    \sum_{0<|k|\leq n-1}
    \left(
        d_k(\lambda)\rho^{|k|}e^{ik\eta}
        +
        d_k^\#(\lambda)\rho^{-|k|}e^{ik\eta}
    \right),
    \qquad
    y=\rho e^{i\eta}.
\]
Equivalently,
\[
\begin{aligned}
    \pi_\infty(u_\lambda)
    &=
    d_0(\lambda)
    +
    d_0^\#(\lambda)\log|y|
    \\
    &\quad
    +
    \sum_{k=1}^{n-1}
    \left(
        d_k(\lambda)y^k
        +
        d_{-k}(\lambda)\bar y^k
        +
        d_{-k}^\#(\lambda)y^{-k}
        +
        d_k^\#(\lambda)\bar y^{-k}
    \right).
\end{aligned}
\]
Suppose that the critical modes at \(\infty\) are written as
\[
    v_0(\rho;\lambda)
    =
    A_0^\infty(\lambda)
    P_\nu\left(\frac{1-\rho^{2n}}{1+\rho^{2n}}\right)
    +
    B_0^\infty(\lambda)
    Q_\nu\left(\frac{1-\rho^{2n}}{1+\rho^{2n}}\right)
\]
and, for \(0<|k|\leq n-1\),
\[
    v_k(\rho;\lambda)
    =
    A_k^\infty(\lambda)
    P_\nu^{-\alpha_k}
    \left(\frac{1-\rho^{2n}}{1+\rho^{2n}}\right)
    +
    B_k^\infty(\lambda)
    P_\nu^{\alpha_k}
    \left(\frac{1-\rho^{2n}}{1+\rho^{2n}}\right).
\]
Applying the preceding asymptotic formulas in the coordinate
\(y=\rho e^{i\eta}\), we obtain
\[
\begin{aligned}
    \pi_\infty(u_\lambda)
    &=
    A_0^\infty(\lambda)
    -
    B_0^\infty(\lambda)
    \bigl(\gamma_E+\psi(\nu+1)\bigr)
    -
    nB_0^\infty(\lambda)\log|y|
    \\
    &\quad
    +
    \sum_{k=1}^{n-1}
    \left[
        \frac{A_k^\infty(\lambda)}{\Gamma(1+k/n)}y^k
        +
        \frac{A_{-k}^\infty(\lambda)}{\Gamma(1+k/n)}
        \bar y^k
    \right]
    \\
    &\quad
    +
    \sum_{k=1}^{n-1}
    \left[
        \frac{B_{-k}^\infty(\lambda)}{\Gamma(1-k/n)}
        y^{-k}
        +
        \frac{B_k^\infty(\lambda)}{\Gamma(1-k/n)}
        \bar y^{-k}
    \right].
\end{aligned}
\]
Equivalently, in terms of the bases \(g_j\) and \(g_j^\#\), we have
\begin{equation}\label{eq:pi-infty-critical-asymptotics}
\begin{aligned}
    \pi_\infty(u_\lambda)
    &=
    \sqrt{2\pi}
    \left[
        A_0^\infty(\lambda)
        -
        B_0^\infty(\lambda)
        \bigl(\gamma_E+\psi(\nu+1)\bigr)
    \right]g_0
    +
    n\sqrt{2\pi}\,B_0^\infty(\lambda)g_0^\#
    \\
    &\quad
    +
    \sum_{0<|k|\leq n-1}
    \sqrt{4\pi|k|}
    \left[
        \frac{A_k^\infty(\lambda)}
             {\Gamma(1+\alpha_k)}g_k
        +
        \frac{B_k^\infty(\lambda)}
             {\Gamma(1-\alpha_k)}g_k^\#
    \right].
\end{aligned}
\end{equation}

We now compare the two boundary descriptions of the same formal solution.
The coordinate change \(y=x^{-1}\) gives \(\rho=r^{-1}\) and
\(\eta=-\theta\). Consequently, the \(k\)-th Fourier mode in the
\(y\)-coordinate is determined by the \((-k)\)-th Fourier mode in the
\(x\)-coordinate. The Legendre coefficients appearing in
\(\pi_\infty(u_\lambda)\) can therefore be expressed in terms of those
appearing in \(\pi_0(u_\lambda)\) by means of the Legendre connection
formulas.

\begin{lem}\label{lem:legendre-coefficient-connection}
Let \(\lambda=\nu(\nu+1)\). The Legendre coefficients of the critical modes
of \(u_\lambda\) at \(\infty\) are related to those at \(0\) as follows.

For the zero mode,
\[
    \begin{pmatrix}
        A_0^\infty(\lambda)\\[0.3em]
        B_0^\infty(\lambda)
    \end{pmatrix}
    =
    M_0(\nu)
    \begin{pmatrix}
        A_0(\lambda)\\[0.3em]
        B_0(\lambda)
    \end{pmatrix},
\]
where
\[
    M_0(\nu)
    =
    \begin{pmatrix}
        \cos\pi\nu
        &
        -\dfrac{\pi}{2}\sin\pi\nu
        \\[0.8em]
        -\dfrac{2}{\pi}\sin\pi\nu
        &
        -\cos\pi\nu
    \end{pmatrix}.
\]

For \(0<|k|\leq n-1\), set \(\alpha_k=|k|/n\). Then
\[
    \begin{pmatrix}
        A_k^\infty(\lambda)\\[0.3em]
        B_k^\infty(\lambda)
    \end{pmatrix}
    =
    M_k(\nu)
    \begin{pmatrix}
        A_{-k}(\lambda)\\[0.3em]
        B_{-k}(\lambda)
    \end{pmatrix},
\]
where
\[
    M_k(\nu)
    =
    \begin{pmatrix}
    \displaystyle
        \frac{\sin\pi\nu}{\sin\pi\alpha_k}
    &
    \displaystyle
        \frac{\Gamma(\nu+\alpha_k+1)}
             {\Gamma(\nu-\alpha_k+1)}
        \frac{\sin\pi(\nu+\alpha_k)}
             {\sin\pi\alpha_k}
    \\[1.2em]
    \displaystyle
        -
        \frac{\Gamma(\nu-\alpha_k+1)}
             {\Gamma(\nu+\alpha_k+1)}
        \frac{\sin\pi(\nu-\alpha_k)}
             {\sin\pi\alpha_k}
    &
    \displaystyle
        -
        \frac{\sin\pi\nu}{\sin\pi\alpha_k}
    \end{pmatrix}.
\]
The entries are understood by analytic continuation at any removable
singularities. Moreover,
\[
    \det M_k(\nu)=-1
    \qquad
    \text{for every } |k|\leq n-1.
\]
\end{lem}

\begin{proof}
For the zero mode, the connection formulas are
\[
    P_\nu(-t)
    =
    \cos\pi\nu\,P_\nu(t)
    -
    \frac{2}{\pi}\sin\pi\nu\,Q_\nu(t)
\]
and
\[
    Q_\nu(-t)
    =
    -
    \frac{\pi}{2}\sin\pi\nu\,P_\nu(t)
    -
    \cos\pi\nu\,Q_\nu(t).
\]
These give the matrix \(M_0(\nu)\).

For \(0<|k|\leq n-1\), write \(\alpha=\alpha_k\). Since
\(0<\alpha<1\), the functions \(P_\nu^{-\alpha}\) and
\(P_\nu^\alpha\) form a fundamental system. Their connection formulas are
\[
\begin{aligned}
    P_\nu^{-\alpha}(-t)
    &=
    \frac{\sin\pi\nu}{\sin\pi\alpha}
    P_\nu^{-\alpha}(t)
    \\
    &\quad
    -
    \frac{\Gamma(\nu-\alpha+1)}
         {\Gamma(\nu+\alpha+1)}
    \frac{\sin\pi(\nu-\alpha)}
         {\sin\pi\alpha}
    P_\nu^\alpha(t)
\end{aligned}
\]
and
\[
\begin{aligned}
    P_\nu^\alpha(-t)
    &=
    \frac{\Gamma(\nu+\alpha+1)}
         {\Gamma(\nu-\alpha+1)}
    \frac{\sin\pi(\nu+\alpha)}
         {\sin\pi\alpha}
    P_\nu^{-\alpha}(t)
    \\
    &\quad
    -
    \frac{\sin\pi\nu}{\sin\pi\alpha}
    P_\nu^\alpha(t).
\end{aligned}
\]
Since the coordinate change \(y=x^{-1}\) sends the \((-k)\)-th Fourier mode
in the \(x\)-coordinate to the \(k\)-th Fourier mode in the \(y\)-coordinate,
these formulas give the matrix \(M_k(\nu)\).

Finally,
\[
    \det M_0(\nu)
    =
    -\cos^2\pi\nu-\sin^2\pi\nu
    =
    -1.
\]
For \(0<|k|\leq n-1\), the gamma factors cancel, and hence
\[
\begin{aligned}
    \det M_k(\nu)
    &=
    \frac{
        -\sin^2\pi\nu
        +
        \sin\pi(\nu+\alpha_k)
        \sin\pi(\nu-\alpha_k)
    }{
        \sin^2\pi\alpha_k
    }
    \\
    &=
    -1,
\end{aligned}
\]
where we have used
\[
    \sin\pi(\nu+\alpha_k)\sin\pi(\nu-\alpha_k)
    =
    \sin^2\pi\nu-\sin^2\pi\alpha_k.
\]
\end{proof}

Let
\(
    E_n
    =
    \bigoplus_{|k|\leq n-1}\mb C^2.
\)
For \(\mf a=(\mf a_k)_{|k|\leq n-1}\in E_n\), write
\(\mf a_k=(A_k,B_k)^t\). Choose \(\nu\) such that
\(\lambda=\nu(\nu+1)\) and
\(\nu\notin\{-1,-2,\ldots\}\). This entails no restriction on \(\lambda\),
since \(\nu\) may be replaced by \(-\nu-1\).

Define \(\mc A_0^\nu:E_n\to V_0\) by
\[
\begin{aligned}
    \mc A_0^\nu(\mf a)
    &=
    \sqrt{2\pi}
    \left[
        A_0
        -
        B_0\bigl(\gamma_E+\psi(\nu+1)\bigr)
    \right]f_0
    +
    n\sqrt{2\pi}\,B_0 f_0^\#
    \\
    &\quad
    +
    \sum_{0<|k|\leq n-1}
    \sqrt{4\pi|k|}
    \left[
        \frac{A_k}{\Gamma(1+\alpha_k)}f_k
        +
        \frac{B_k}{\Gamma(1-\alpha_k)}f_k^\#
    \right].
\end{aligned}
\]
Similarly, define \(\mc A_\infty^\nu:E_n\to V_\infty\) by
\[
\begin{aligned}
    \mc A_\infty^\nu(\mf a)
    &=
    \sqrt{2\pi}
    \left[
        A_0
        -
        B_0\bigl(\gamma_E+\psi(\nu+1)\bigr)
    \right]g_0
    +
    n\sqrt{2\pi}\,B_0 g_0^\#
    \\
    &\quad
    +
    \sum_{0<|k|\leq n-1}
    \sqrt{4\pi|k|}
    \left[
        \frac{A_k}{\Gamma(1+\alpha_k)}g_k
        +
        \frac{B_k}{\Gamma(1-\alpha_k)}g_k^\#
    \right].
\end{aligned}
\]

\begin{lem}\label{lem:boundary-connection-map}
Define a linear map \(\mc M_\nu:E_n\to E_n\) by
\[
    (\mc M_\nu\mf a)_0
    =
    M_0(\nu)\mf a_0
\]
and, for \(0<|k|\leq n-1\), by
\[
    (\mc M_\nu\mf a)_k
    =
    M_k(\nu)\mf a_{-k}.
\]
Then
\[
    \mc C_\nu
    =
    \mc A_\infty^\nu
    \circ
    \mc M_\nu
    \circ
    (\mc A_0^\nu)^{-1}
    :
    V_0\longrightarrow V_\infty
\]
is a linear isomorphism. Moreover, every formal solution \(u_\lambda\) of
\eqref{eq:model}, with \(\lambda=\nu(\nu+1)\), satisfies
\[
    \pi_\infty(u_\lambda)
    =
    \mc C_\nu\pi_0(u_\lambda).
\]
\end{lem}

\begin{proof}
The maps \(\mc A_0^\nu\) and \(\mc A_\infty^\nu\) are isomorphisms. Indeed,
their matrices in the chosen bases are block diagonal with respect to the
Fourier modes, except for the triangular zero-mode block, and all their
diagonal entries are nonzero.

The map
\[
    (\mf a_k)_{|k|\leq n-1}
    \longmapsto
    (\mf a_{-k})_{|k|\leq n-1}
\]
is a permutation of the summands of \(E_n\). By
Lemma~\ref{lem:legendre-coefficient-connection}, each matrix \(M_k(\nu)\)
is invertible. Hence \(\mc M_\nu\) is an isomorphism, and therefore so is
\(\mc C_\nu\).

If \(u_\lambda\) has Legendre coefficient vector \(\mf a\) at \(0\), then
\[
    \pi_0(u_\lambda)=\mc A_0^\nu(\mf a).
\]
Lemma~\ref{lem:legendre-coefficient-connection} shows that its Legendre
coefficient vector at \(\infty\) is \(\mc M_\nu\mf a\). Therefore
\[
\begin{aligned}
    \pi_\infty(u_\lambda)
    &=
    \mc A_\infty^\nu(\mc M_\nu\mf a)
    \\
    &=
    \mc A_\infty^\nu
    \circ
    \mc M_\nu
    \circ
    (\mc A_0^\nu)^{-1}
    \pi_0(u_\lambda)
    \\
    &=
    \mc C_\nu\pi_0(u_\lambda).
\end{aligned}
\]
\end{proof}

We conclude this section by recording the corresponding boundary map for a
general element of the maximal domain. Let
\(u\in\mc D_{\max}(\Delta)\). Then \(\Delta u=F\) for some
\(F\in L^2(X',dA)\) in the sense of distributions. By the general
asymptotic decomposition for conic Laplacians, \(u\) admits critical
asymptotic expansions at the two conic points. We define
\(\pi_0(u)\in V_0\) and \(\pi_\infty(u)\in V_\infty\) by retaining the
leading critical asymptotic terms in the coordinates \(x\) and \(y\),
respectively.

We use the same vector space
\(E_n=\bigoplus_{|k|\leq n-1}\mb C^2\), now regarded as the boundary
coefficient space. For
\(\mf a=(\mf a_k)_{|k|\leq n-1}\in E_n\), write
\(\mf a_k=(a_k,a_k^\#)^t\), and define
\[
    \mc A_0:E_n\longrightarrow V_0,
    \qquad
    \mc A_0(\mf a)
    =
    \sum_{|k|\leq n-1}
    \left(
        a_kf_k+a_k^\#f_k^\#
    \right).
\]
Similarly, for
\(\mf b=(\mf b_k)_{|k|\leq n-1}\in E_n\), with
\(\mf b_k=(b_k,b_k^\#)^t\), define
\[
    \mc A_\infty:E_n\longrightarrow V_\infty,
    \qquad
    \mc A_\infty(\mf b)
    =
    \sum_{|k|\leq n-1}
    \left(
        b_kg_k+b_k^\#g_k^\#
    \right).
\]
The maps \(\mc A_0\) and \(\mc A_\infty\) are linear isomorphisms.
Consequently, there exist unique \(\mf a,\mf b\in E_n\) such that
\[
    \pi_0(u)=\mc A_0(\mf a),
    \qquad
    \pi_\infty(u)=\mc A_\infty(\mf b).
\]
Thus \(\mf a\) and \(\mf b\) are the normalized boundary coefficient
vectors of the critical asymptotic expansions of \(u\) at \(0\) and
\(\infty\), respectively.

Define
\[
    \pi_S
    =
    \pi_0\oplus\pi_\infty:
    \mc D_{\max}(\Delta)
    \longrightarrow
    V_S.
\]
The general asymptotic decomposition of the maximal domain and the
construction of the boundary map for conic Laplacians were established in
\cite[Section~2]{LiouKrein}. Applied to the present model, those results imply
that \(\pi_S\) is surjective and that
\[
    \ker\pi_S=\mc D_{\min}(\Delta).
\]
Consequently, \(\pi_S\) induces a linear isomorphism
\[
    \bar\pi_S:
    \mc D_{\max}(\Delta)/\mc D_{\min}(\Delta)
    \longrightarrow
    V_S,
    \qquad
    \bar\pi_S([u])=\pi_S(u).
\]
Our purpose here is not to repeat the general construction, but to identify
the boundary map explicitly for the projective model.

The preceding Legendre connection formulas now admit a direct interpretation
in terms of the boundary data spaces. Although connection matrices between
local bases of solutions are classical in the theory of Fuchsian and
hypergeometric equations, in the present setting they acquire a concrete
geometric meaning. The maps \(\mc A_0^\nu\) and
\(\mc A_\infty^\nu\) identify the mode-wise Legendre coefficient space
\(E_n\) with the boundary data spaces \(V_0\) and \(V_\infty\) determined by
the critical asymptotic expansions at the conic points \(0\) and
\(\infty\), respectively.

The map \(\mc M_\nu\) is the mode-wise connection map relating the
coefficients of the same formal solution in the two affine coordinates. The
appearance of \(\mf a_{-k}\) reflects the reversal of angular modes under the
coordinate transformation \(y=x^{-1}\). Therefore
\[
    \mc C_\nu
    =
    \mc A_\infty^\nu
    \circ
    \mc M_\nu
    \circ
    (\mc A_0^\nu)^{-1}
\]
is the induced map from the boundary data at \(0\) to those at \(\infty\).
It specifies how the boundary data at the two conic points must be matched in
order to arise from a single formal solution of the model equation on
\(\mb P^1\). In this sense, the classical Legendre connection formulas are
realized in the projective conic model as geometric gluing laws for critical
asymptotic boundary data.

\section{Self-Adjoint Extensions of the Laplacian and the Friedrichs Spectrum}

In this section, we describe the self-adjoint extensions of the Laplacian and
compute the spectrum of its Friedrichs extension.

The Green form is the skew-Hermitian sesquilinear form
\[
    \mk q:
    \mc D_{\max}(\Delta)\times\mc D_{\max}(\Delta)
    \longrightarrow
    \mb C
\]
defined by
\[
    \mk q(u,v)
    =
    \langle\Delta u,v\rangle_{L^2(X',dA)}
    -
    \langle u,\Delta v\rangle_{L^2(X',dA)}.
\]
Its radical is
\[
    \operatorname{rad}\mk q
    =
    \left\{
        v\in\mc D_{\max}(\Delta):
        \mk q(u,v)=0
        \text{ for every }
        u\in\mc D_{\max}(\Delta)
    \right\},
\]
which is precisely \(\mc D_{\min}(\Delta)\). Hence \(\mk q\) descends to a
nondegenerate skew-Hermitian form on
\(\mc D_{\max}(\Delta)/\mc D_{\min}(\Delta)\).

Under the identification
\[
    \bar\pi_S:
    \mc D_{\max}(\Delta)/\mc D_{\min}(\Delta)
    \longrightarrow
    V_S
\]
induced by \(\pi_S\), we denote the resulting form on \(V_S\) by
\(\omega_S\).

With the normalizations of the critical asymptotic bases chosen above, the
induced forms on \(V_0\) and \(V_\infty\) are characterized by
\[
    \omega_0(f_j,f_k^\#)=\delta_{jk},
    \qquad
    \omega_0(f_j,f_k)
    =
    \omega_0(f_j^\#,f_k^\#)
    =
    0,
\]
and
\[
    \omega_\infty(g_j,g_k^\#)=\delta_{jk},
    \qquad
    \omega_\infty(g_j,g_k)
    =
    \omega_\infty(g_j^\#,g_k^\#)
    =
    0,
\]
for \(|j|,|k|\leq n-1\), with the remaining values determined by
skew-Hermitianity. In particular,
\[
    \omega_0(f_j^\#,f_k)=-\delta_{jk},
    \qquad
    \omega_\infty(g_j^\#,g_k)=-\delta_{jk}.
\]
Under the decomposition \(V_S=V_0\oplus V_\infty\), we have
\[
    \omega_S=\omega_0\oplus\omega_\infty.
\]
We call \(\omega_S\) the boundary symplectic form on \(V_S\).

For \(u,v\in C^\infty(X')\cap\mc D_{\max}(\Delta)\), Green's formula gives
\[
    \mk q(u,v)
    =
    \omega_0\bigl(\pi_0(u),\pi_0(v)\bigr)
    +
    \omega_\infty\bigl(\pi_\infty(u),\pi_\infty(v)\bigr).
\]
Both sides are continuous with respect to the graph norm on
\(\mc D_{\max}(\Delta)\). Hence this identity extends to all
\(u,v\in\mc D_{\max}(\Delta)\). Equivalently,
\[
    \mk q(u,v)
    =
    \omega_S\bigl(\pi_S(u),\pi_S(v)\bigr)
\]
for all \(u,v\in\mc D_{\max}(\Delta)\).

\begin{defn}
For a subspace \(\mc V\subset V_S\), its symplectic orthogonal complement is
\[
    \mc V^{\perp_{\omega_S}}
    =
    \bigl\{
        \xi\in V_S:
        \omega_S(\xi,\eta)=0
        \text{ for every }\eta\in\mc V
    \bigr\}.
\]
A subspace \(\mc V\subset V_S\) is called Lagrangian if
\(\mc V=\mc V^{\perp_{\omega_S}}\). We denote by
\(\operatorname{LGr}(V_S,\omega_S)\) the Lagrangian Grassmannian of
\((V_S,\omega_S)\), that is, the set of all Lagrangian subspaces of
\(V_S\).
\end{defn}

In the previous section, we introduced the subspaces \(\mc F_P\) and
\(\mc F_P^\#\) for \(P\in S\). Set
\(\mc F=\mc F_0\oplus\mc F_\infty\) and
\(\mc F^\#=\mc F_0^\#\oplus\mc F_\infty^\#\). Then both
\(\mc F\) and \(\mc F^\#\) belong to
\(\operatorname{LGr}(V_S,\omega_S)\), and they are complementary:
\[
    V_S=\mc F\oplus\mc F^\#.
\]
We call \((\mc F,\mc F^\#)\) a Lagrangian pair and refer to
\(\mc F\) as the Friedrichs Lagrangian subspace of \(V_S\). More generally,
a pair \((\mc V,\mc V^\#)\) of Lagrangian subspaces of
\((V_S,\omega_S)\) is called a Lagrangian pair if
\(V_S=\mc V\oplus\mc V^\#\).

We now describe the self-adjoint extensions of
\((\Delta_{\min},\mc D_{\min}(\Delta))\), where
\(\Delta_{\min}=\Delta|_{\mc D_{\min}(\Delta)}\). For each
\(\mc V\in\operatorname{LGr}(V_S,\omega_S)\), define
\[
    \mc D_{\mc V}
    =
    \bigl\{
        u\in\mc D_{\max}(\Delta):
        \pi_S(u)\in\mc V
    \bigr\},
\]
and set \(\Delta_{\mc V}=\Delta|_{\mc D_{\mc V}}\). Since \(\mc V\) is Lagrangian,
\((\Delta_{\mc V},\mc D_{\mc V})\) is a self-adjoint extension of
\((\Delta_{\min},\mc D_{\min}(\Delta))\).

Conversely, let \((\widetilde\Delta,\mc D)\) be a self-adjoint extension of
\((\Delta_{\min},\mc D_{\min}(\Delta))\), with
\(\mc D\subset\mc D_{\max}(\Delta)\). Then \(\pi_S(\mc D)\) is a
Lagrangian subspace of \((V_S,\omega_S)\). Moreover, since
\(\ker\pi_S=\mc D_{\min}(\Delta)\), there exists a unique
\(\mc V\in\operatorname{LGr}(V_S,\omega_S)\) such that
\(\mc D=\mc D_{\mc V}\), and hence
\((\widetilde\Delta,\mc D)
=(\Delta_{\mc V},\mc D_{\mc V})\).
Thus the self-adjoint extensions of
\((\Delta_{\min},\mc D_{\min}(\Delta))\) are in bijection with
\(\operatorname{LGr}(V_S,\omega_S)\). For the general theory of closed extensions of elliptic cone operators, see
\cite{GilMendoza,GilKrainerMendoza}.

By standard elliptic theory for conic operators, each
\(\Delta_{\mc V}\) has compact resolvent. Consequently, its spectrum consists
of real eigenvalues of finite multiplicity and has no finite accumulation
point. Moreover, since \(\Delta_{\min}\) is nonnegative and
\(\mc D_{\max}(\Delta)/\mc D_{\min}(\Delta)\) is finite-dimensional, every
self-adjoint extension \(\Delta_{\mc V}\) is bounded from below; see
\cite{GilKrainerMendoza}.

Lemma~\ref{lem:large-mode-regularity-condition} shows that, for
\(|k|\geq n\), a nonzero square-integrable formal solution in the \(k\)-th
mode exists only when, up to the symmetry \(\nu\mapsto-\nu-1\), one has
\(\nu=\alpha_k+q\) for some \(q\in\mb Z_{\geq0}\). For
\(k\in\mb Z\) with \(|k|\geq n\), define
\(\Psi_{q,k}\in\mc D_{\max}(\Delta)\) by
\[
    \Psi_{q,k}(r,\theta)
    =
    \Psi_k(r,\theta;\lambda_{q,k}),
    \qquad
    \lambda_{q,k}
    =
    (\alpha_k+q)(\alpha_k+q+1).
\]
The functions \(\Psi_{q,k}\) have no critical asymptotic boundary terms at
either conic point. Hence \(\pi_S(\Psi_{q,k})=0\), and therefore
\(\Psi_{q,k}\in\mc D_{\min}(\Delta)\). Since
\(\mc D_{\min}(\Delta)\subset\mc D_{\mc V}\) for every
\(\mc V\in\operatorname{LGr}(V_S,\omega_S)\), each \(\Psi_{q,k}\) is an
eigenfunction of every self-adjoint extension \(\Delta_{\mc V}\), with
eigenvalue \(\lambda_{q,k}\).

To determine the remaining part of the Friedrichs spectrum
\(\sigma(\Delta_{\mc F})\), it remains to consider the formal solutions
\(u_\lambda\) of \eqref{eq:model} satisfying the Friedrichs boundary
condition \(\pi_S(u_\lambda)\in\mc F\). Since
\(\mc F=\mc F_0\oplus\mc F_\infty\), this condition is equivalent to
\(\pi_0(u_\lambda)\in\mc F_0\) and
\(\pi_\infty(u_\lambda)\in\mc F_\infty\). By
\eqref{eq:pi0-critical-asymptotics} and
\eqref{eq:pi-infty-critical-asymptotics}, the Friedrichs boundary condition
is equivalent to
\[
    B_k(\lambda)=B_k^\infty(\lambda)=0
    \qquad
    \text{for all } |k|\leq n-1.
\]

For the zero mode, the Friedrichs boundary condition gives
\(B_0(\lambda)=B_0^\infty(\lambda)=0\). Since
\[
    \begin{pmatrix}
        A_0^\infty(\lambda)\\
        B_0^\infty(\lambda)
    \end{pmatrix}
    =
    M_0(\nu)
    \begin{pmatrix}
        A_0(\lambda)\\
        0
    \end{pmatrix},
\]
we obtain
\(A_0^\infty(\lambda)=\cos\pi\nu\,A_0(\lambda)\) and
\(B_0^\infty(\lambda)=-(2/\pi)\sin\pi\nu\,A_0(\lambda)\). Hence, for a
nonzero zero-mode solution, the condition \(B_0^\infty(\lambda)=0\) is
equivalent to \(\sin\pi\nu=0\). Up to the symmetry
\(\nu\mapsto-\nu-1\), we may therefore write
\(\nu=q\in\mb Z_{\geq0}\). Thus
\(\lambda_{q,0}=q(q+1)\), and
\[
    \Psi_{q,0}(r,\theta)
    =
    P_q\left(\frac{1-r^{2n}}{1+r^{2n}}\right),
    \qquad
    q\in\mb Z_{\geq0},
\]
is an eigenfunction of \(\Delta_{\mc F}\) with eigenvalue
\(\lambda_{q,0}\).

Now let \(m\in\mb Z\) satisfy \(0<|m|\leq n-1\). The Friedrichs boundary
condition gives \(B_m(\lambda)=0\) and
\(B_{-m}^\infty(\lambda)=0\). Since the \((-m)\)-mode at \(\infty\) is
related to the \(m\)-mode at \(0\) by
\[
    \begin{pmatrix}
        A_{-m}^\infty(\lambda)\\
        B_{-m}^\infty(\lambda)
    \end{pmatrix}
    =
    M_m(\nu)
    \begin{pmatrix}
        A_m(\lambda)\\
        0
    \end{pmatrix},
\]
we have
\[
    B_{-m}^\infty(\lambda)
    =
    -
    \frac{\Gamma(\nu-\alpha_m+1)}
         {\Gamma(\nu+\alpha_m+1)}
    \frac{\sin\pi(\nu-\alpha_m)}
         {\sin\pi\alpha_m}
    A_m(\lambda).
\]
By the reflection formula for the gamma function, the coefficient on the
right-hand side can be rewritten as
\[
    \frac{\pi}
    {
        \sin\pi\alpha_m\,
        \Gamma(\alpha_m-\nu)
        \Gamma(\nu+\alpha_m+1)
    }.
\]
Therefore, for \(A_m(\lambda)\neq0\), the condition
\(B_{-m}^\infty(\lambda)=0\) is equivalent to
\[
    \frac{1}
    {
        \Gamma(\alpha_m-\nu)
        \Gamma(\nu+\alpha_m+1)
    }
    =
    0.
\]
Since the gamma function has poles precisely at the nonpositive integers,
this holds if and only if
\[
    \nu=\alpha_m+q
    \qquad\text{or}\qquad
    \nu=-\alpha_m-q-1
\]
for some \(q\in\mb Z_{\geq0}\). These two possibilities are related by the
symmetry \(\nu\mapsto-\nu-1\). Hence, up to this symmetry, we may write
\(\nu=\alpha_m+q\), and therefore
\[
    \lambda_{q,m}
    =
    (\alpha_m+q)(\alpha_m+q+1).
\]
Relabeling \(m\) as \(k\), the corresponding Friedrichs eigenfunctions are
\[
    \Psi_{q,k}(r,\theta)
    =
    P_{\alpha_k+q}^{-\alpha_k}
    \left(\frac{1-r^{2n}}{1+r^{2n}}\right)e^{ik\theta},
    \qquad
    0<|k|\leq n-1,\quad q\in\mb Z_{\geq0}.
\]
They belong to \(\mc D_{\mc F}\) and satisfy
\[
    \Delta_{\mc F}\Psi_{q,k}
    =
    \lambda_{q,k}\Psi_{q,k},
    \qquad
    \lambda_{q,k}
    =
    (\alpha_k+q)(\alpha_k+q+1).
\]

We conclude with the following description of the Friedrichs spectrum.

\begin{thm}\label{thm:friedrichs-spectrum}
For \(k\in\mb Z\), set \(\alpha_k=|k|/n\), and, for \(k\neq0\), let
\(\mu_k\) be as defined in Section~1. The spectrum of
\(\Delta_{\mc F}\), as a set, is
\[
    \sigma(\Delta_{\mc F})
    =
    \left\{
        \lambda_{q,k}
        =
        (\alpha_k+q)(\alpha_k+q+1)
        :
        q\in\mb Z_{\geq0},\ k\in\mb Z
    \right\}.
\]
For \(k=0\), an eigenfunction corresponding to \(\lambda_{q,0}=q(q+1)\)
is
\[
    \Psi_{q,0}(r,\theta)
    =
    P_q\left(\frac{1-r^{2n}}{1+r^{2n}}\right),
    \qquad
    q\in\mb Z_{\geq0}.
\]
For \(k\neq0\), an eigenfunction corresponding to
\(\lambda_{q,k}\) is
\[
    \Psi_{q,k}(r,\theta)
    =
    P_{\alpha_k+q}^{\mu_k}
    \left(\frac{1-r^{2n}}{1+r^{2n}}\right)e^{ik\theta},
    \qquad
    q\in\mb Z_{\geq0}.
\]
Moreover, the family
\[
    \bigl\{
        \Psi_{q,k}:
        q\in\mb Z_{\geq0},\ k\in\mb Z
    \bigr\}
\]
is an orthogonal system in \(L^2(X',dA)\), and its linear span is dense in
\(L^2(X',dA)\).
\end{thm}

The completeness of the preceding eigenfunctions follows from separation of
variables. The Hilbert space \(L^2(X',dA)\) admits the orthogonal
decomposition
\[
    L^2(X',dA)
    =
    \widehat{\bigoplus}_{k\in\mb Z}H_k,
\]
where
\[
    H_k
    =
    \bigl\{
        u_k(r)e^{ik\theta}:
        u_k\in L^2_{\mathrm{rad}}
    \bigr\}
\]
and
\[
    L^2_{\mathrm{rad}}
    =
    L^2\left(
        (0,\infty),
        \frac{4n^2r^{2n-1}}{(1+r^{2n})^2}\,dr
    \right).
\]
The Friedrichs Laplacian reduces each Fourier subspace \(H_k\). For each
fixed \(k\), its radial part is a one-dimensional Friedrichs
Sturm--Liouville operator. Under the change of variables
\[
    t=\frac{1-r^{2n}}{1+r^{2n}},
\]
the radial operator becomes
\[
    -
    \frac{d}{dt}
    \left(
        (1-t^2)\frac{d}{dt}
    \right)
    +
    \frac{\alpha_k^2}{1-t^2}
\]
on \(L^2((-1,1),dt)\).

For \(k=0\), the Friedrichs eigenfunctions of the transformed radial
operator are \(P_q(t)\), \(q\in\mb Z_{\geq0}\). For \(k\neq0\), they are
\(P_{\alpha_k+q}^{\mu_k}(t)\), \(q\in\mb Z_{\geq0}\). By the spectral
theorem for the corresponding one-dimensional Friedrichs
Sturm--Liouville operator, these eigenfunctions form a complete orthogonal
system in the \(k\)-th radial space. Taking the orthogonal sum over all
Fourier modes, we conclude that
\[
    \bigl\{
        \Psi_{q,k}:
        q\in\mb Z_{\geq0},\ k\in\mb Z
    \bigr\}
\]
is a complete orthogonal system in \(L^2(X',dA)\).

We next compute the \(L^2(X',dA)\)-norms of these eigenfunctions.

\begin{lem}\label{lem:radial-change-of-variables}
Let \(F:[-1,1]\to\mb C\) be continuous. Then
\[
    \int_0^\infty
    F\left(\frac{1-r^{2n}}{1+r^{2n}}\right)
    \frac{4n^2r^{2n-1}}{(1+r^{2n})^2}\,dr
    =
    n\int_{-1}^1F(t)\,dt.
\]
\end{lem}
\begin{proof}
This follows from the change of variables
\(t=(1-r^{2n})/(1+r^{2n})\).
\end{proof}

\begin{lem}\label{lem:norm-of-friedrichs-eigenfunctions}
The \(L^2(X',dA)\)-norm of \(\Psi_{q,k}\) is
\[
    \|\Psi_{q,k}\|_{L^2(X',dA)}^2
    =
    \begin{cases}
    \displaystyle
        \frac{4\pi n}{2q+1},
        & k=0,
        \\[1em]
    \displaystyle
        \frac{4\pi n}{2q+2\alpha_k+1}
        \frac{
            \Gamma(\alpha_k+q+\mu_k+1)
        }{
            \Gamma(\alpha_k+q-\mu_k+1)
        },
        & k\neq0.
    \end{cases}
\]
\end{lem}

\begin{proof}
Define the radial part of \(\Psi_{q,k}\) by
\[
    \Phi_{q,k}(r)
    =
    \begin{cases}
    \displaystyle
        P_q\left(\frac{1-r^{2n}}{1+r^{2n}}\right),
        & k=0,
        \\[1em]
    \displaystyle
        P_{\alpha_k+q}^{\mu_k}
        \left(\frac{1-r^{2n}}{1+r^{2n}}\right),
        & k\neq0.
    \end{cases}
\]
Thus \(\Psi_{q,k}(r,\theta)=\Phi_{q,k}(r)e^{ik\theta}\). Since
\(|e^{ik\theta}|=1\), we have
\[
\begin{aligned}
    \|\Psi_{q,k}\|_{L^2(X',dA)}^2
    &=
    \int_{X'}|\Psi_{q,k}|^2\,dA
    \\
    &=
    2\pi
    \int_0^\infty
    |\Phi_{q,k}(r)|^2
    \frac{4n^2r^{2n-1}}{(1+r^{2n})^2}\,dr.
\end{aligned}
\]
By Lemma~\ref{lem:radial-change-of-variables},
\[
    \|\Psi_{q,k}\|_{L^2(X',dA)}^2
    =
    \begin{cases}
    \displaystyle
        2\pi n\int_{-1}^1|P_q(t)|^2\,dt,
        & k=0,
        \\[1em]
    \displaystyle
        2\pi n\int_{-1}^1
        \left|
            P_{\alpha_k+q}^{\mu_k}(t)
        \right|^2\,dt,
        & k\neq0.
    \end{cases}
\]
Using
\[
    \int_{-1}^1|P_q(t)|^2\,dt
    =
    \frac{2}{2q+1}
\]
and, for \(k\neq0\),
\[
    \int_{-1}^1
    \left|
        P_{\alpha_k+q}^{\mu_k}(t)
    \right|^2\,dt
    =
    \frac{2}{2q+2\alpha_k+1}
    \frac{
        \Gamma(\alpha_k+q+\mu_k+1)
    }{
        \Gamma(\alpha_k+q-\mu_k+1)
    },
\]
we obtain the asserted formula.
\end{proof}

\section{Formal Solutions on the Resolvent Set and the Weyl Function}

Let \((\mc V,\mc V^\#)\) be a Lagrangian pair in
\((V_S,\omega_S)\). Let \(P_{\mc V}:V_S\to\mc V\) and
\(P_{\mc V^\#}:V_S\to\mc V^\#\) be the projections associated with the
decomposition \(V_S=\mc V\oplus\mc V^\#\). Define
\[
    \widetilde\pi_{\mc V}
    =
    P_{\mc V}\circ\pi_S:
    \mc D_{\max}(\Delta)\longrightarrow\mc V,
    \qquad
    \widetilde\pi_{\mc V^\#}
    =
    P_{\mc V^\#}\circ\pi_S:
    \mc D_{\max}(\Delta)\longrightarrow\mc V^\#.
\]

Let \(\lambda\in\rho(\Delta_{\mc V})\) and
\(h^\#\in\mc V^\#\). We seek a formal solution
\(u_\lambda\in\mc D_{\max}(\Delta)\) of \eqref{eq:model} satisfying
\(\widetilde\pi_{\mc V^\#}(u_\lambda)=h^\#\).

Such a solution is unique. Indeed, suppose that \(u_\lambda\) and
\(v_\lambda\) are two formal solutions satisfying
\(\widetilde\pi_{\mc V^\#}(u_\lambda)
=\widetilde\pi_{\mc V^\#}(v_\lambda)=h^\#\). Then
\(w_\lambda=u_\lambda-v_\lambda\) satisfies
\((\Delta-\lambda)w_\lambda=0\) and
\(\widetilde\pi_{\mc V^\#}(w_\lambda)=0\). Hence
\(\pi_S(w_\lambda)\in\mc V\), so
\(w_\lambda\in\mc D_{\mc V}\). Therefore
\((\Delta_{\mc V}-\lambda)w_\lambda=0\). Since
\(\lambda\in\rho(\Delta_{\mc V})\), it follows that
\(w_\lambda=0\), and hence \(u_\lambda=v_\lambda\).

We now turn to the existence of such a formal solution. The idea is to start
with an arbitrary element having the prescribed boundary data and then use
the resolvent to remove its failure to satisfy the homogeneous equation.

Since \(\pi_S\) is surjective, so is
\(\widetilde\pi_{\mc V^\#}\). Hence there exists
\(u\in\mc D_{\max}(\Delta)\) such that
\(\widetilde\pi_{\mc V^\#}(u)=h^\#\). To see how the desired solution should
be constructed, let \(u,u'\in\mc D_{\max}(\Delta)\) be two elements
satisfying
\[
    \widetilde\pi_{\mc V^\#}(u)
    =
    \widetilde\pi_{\mc V^\#}(u')
    =
    h^\#.
\]
Since \(u\) and \(u'\) have the same \(\mc V^\#\)-boundary component, their
difference satisfies
\(\widetilde\pi_{\mc V^\#}(u-u')=0\). Hence
\(\pi_S(u-u')\in\mc V\), and therefore
\(u-u'\in\mc D_{\mc V}\). It follows that
\[
    (\Delta-\lambda)u-(\Delta-\lambda)u'
    =
    (\Delta_{\mc V}-\lambda)(u-u').
\]
Applying \(R_{\mc V}(\lambda)=(\lambda-\Delta_{\mc V})^{-1}\) to this
identity, we obtain
\[
    R_{\mc V}(\lambda)(\Delta-\lambda)u
    -
    R_{\mc V}(\lambda)(\Delta-\lambda)u'
    =
    -u+u'.
\]
Consequently,
\[
    u+R_{\mc V}(\lambda)(\Delta-\lambda)u
    =
    u'+R_{\mc V}(\lambda)(\Delta-\lambda)u'.
\]
This identity shows that the expression
\[
    u+R_{\mc V}(\lambda)(\Delta-\lambda)u
\]
depends only on the prescribed boundary value \(h^\#\), and not on the
choice of the lift \(u\in\mc D_{\max}(\Delta)\). This naturally leads us to
define
\[
    u_\lambda
    =
    u+R_{\mc V}(\lambda)(\Delta-\lambda)u,
\]
where \(u\) is any element satisfying
\(\widetilde\pi_{\mc V^\#}(u)=h^\#\).

We now verify that \(u_\lambda\) is the desired formal solution. Since
\(R_{\mc V}(\lambda)(\Delta-\lambda)u\in\mc D_{\mc V}\), its
\(\mc V^\#\)-boundary component vanishes. Thus
\[
    \widetilde\pi_{\mc V^\#}(u_\lambda)
    =
    \widetilde\pi_{\mc V^\#}(u)
    =
    h^\#.
\]
Moreover,
\[
\begin{aligned}
    (\Delta-\lambda)u_\lambda
    &=
    (\Delta-\lambda)u
    +
    (\Delta_{\mc V}-\lambda)
    R_{\mc V}(\lambda)(\Delta-\lambda)u
    \\
    &=
    (\Delta-\lambda)u-(\Delta-\lambda)u
    =
    0.
\end{aligned}
\]
Therefore \(u_\lambda\) is the required formal solution of
\eqref{eq:model} with prescribed \(\mc V^\#\)-boundary component \(h^\#\).
Together with the uniqueness proved above, this shows that such a formal
solution exists uniquely.

This observation provides a natural projection onto the space of formal
solutions. For each \(\lambda\in\rho(\Delta_{\mc V})\), define
\[
    Y_{\mc V}(\lambda):
    \mc D_{\max}(\Delta)\longrightarrow\mc D_{\max}(\Delta)
\]
by
\(
    Y_{\mc V}(\lambda)u
    =
    u+R_{\mc V}(\lambda)(\Delta-\lambda)u.
\)
By the defining property of the resolvent,
\(Y_{\mc V}(\lambda)\) annihilates precisely the elements of
\(\mc D_{\mc V}\), while its image consists exactly of the formal solutions
of \eqref{eq:model}. Consequently,
\(\ker Y_{\mc V}(\lambda)=\mc D_{\mc V}\) and
\(\operatorname{Im}Y_{\mc V}(\lambda)=\ker(\Delta-\lambda)\).

Moreover, \(Y_{\mc V}(\lambda)\) acts as the identity on
\(\ker(\Delta-\lambda)\). Since its image is
\(\ker(\Delta-\lambda)\), it follows that \(Y_{\mc V}(\lambda)\) is a
projection. It therefore induces a linear isomorphism
\[
    \overline{Y}_{\mc V}(\lambda):
    \mc D_{\max}(\Delta)/\mc D_{\mc V}
    \longrightarrow
    \ker(\Delta-\lambda).
\]

Similarly, since
\(\widetilde{\pi}_{\mc V^\#}:
\mc D_{\max}(\Delta)\to\mc V^\#\)
is surjective with kernel \(\mc D_{\mc V}\), it induces a linear
isomorphism
\[
    \overline{\widetilde{\pi}}_{\mc V^\#}:
    \mc D_{\max}(\Delta)/\mc D_{\mc V}
    \longrightarrow
    \mc V^\#.
\]
Combining these two identifications yields the linear isomorphism
\[
    \Gamma_{\mc V,\mc V^\#}(\lambda):
    \mc V^\#
    \longrightarrow
    \ker(\Delta-\lambda),
    \qquad
    \Gamma_{\mc V,\mc V^\#}(\lambda)
    =
    \overline{Y}_{\mc V}(\lambda)
    \circ
    \overline{\widetilde{\pi}}_{\mc V^\#}^{-1}.
\]
Thus \(\Gamma_{\mc V,\mc V^\#}(\lambda)\) identifies prescribed
\(\mc V^\#\)-boundary data with the corresponding formal solution in
\(\ker(\Delta-\lambda)\). In particular,
\(\dim_{\mb C}\ker(\Delta-\lambda)=\dim_{\mb C}\mc V^\#\).
Since this space is finite-dimensional, \(Y_{\mc V}(\lambda)\) is a
finite-rank projection.

Under this identification, a formal solution is determined by its
\(\mc V^\#\)-boundary component, while its complementary
\(\mc V\)-boundary component depends linearly on the prescribed
\(\mc V^\#\)-component. This dependence is encoded by the Weyl function.

\begin{defn}
Let \((\mc V,\mc V^\#)\) be a Lagrangian pair and let
\(\lambda\in\rho(\Delta_{\mc V})\). The Weyl function associated with the
Lagrangian pair \((\mc V,\mc V^\#)\) is the linear map
\(M_{\mc V,\mc V^\#}(\lambda):\mc V^\#\to\mc V\) defined by
\[
    M_{\mc V,\mc V^\#}(\lambda)
    =
    \left.
        \widetilde{\pi}_{\mc V}
    \right|_{\ker(\Delta-\lambda)}
    \circ
    \Gamma_{\mc V,\mc V^\#}(\lambda).
\]
Equivalently, for \(h^\#\in\mc V^\#\),
\[
    M_{\mc V,\mc V^\#}(\lambda)h^\#
    =
    \widetilde{\pi}_{\mc V}
    \bigl(
        \Gamma_{\mc V,\mc V^\#}(\lambda)h^\#
    \bigr).
\]
\end{defn}

Since
\(
    \widetilde{\pi}_{\mc V^\#}
    \bigl(
        \Gamma_{\mc V,\mc V^\#}(\lambda)h^\#
    \bigr)
    =
    h^\#,
\)
it follows directly from the definition that, for every
\(h^\#\in\mc V^\#\),
\[
    \pi_S
    \bigl(
        \Gamma_{\mc V,\mc V^\#}(\lambda)h^\#
    \bigr)
    =
    M_{\mc V,\mc V^\#}(\lambda)h^\#
    +
    h^\#.
\]

\begin{defn}
Let \((\mc V,\mc V^\#)\) be a Lagrangian pair, and let
\(\lambda\in\rho(\Delta_{\mc V})\). Choose symplectically dual ordered
bases \(\beta=(e_1,\ldots,e_d)\) of \(\mc V\) and
\(\beta^\#=(e_1^\#,\ldots,e_d^\#)\) of \(\mc V^\#\), so that
\(\omega_S(e_i,e_j^\#)=\delta_{ij}\). Write
\[
    M_{\mc V,\mc V^\#}(\lambda)e_j^\#
    =
    \sum_{k=1}^d
    m_{kj}^{\mc V,\mc V^\#}(\lambda)e_k.
\]
The matrix
\[
    S_{\mc V,\mc V^\#}^{\beta,\beta^\#}(\lambda)
    :=
    \bigl[
        M_{\mc V,\mc V^\#}(\lambda)
    \bigr]_{\beta^\#}^{\beta}
    =
    \bigl[
        m_{kj}^{\mc V,\mc V^\#}(\lambda)
    \bigr]_{k,j=1}^d
\]
is called the \(S\)-matrix associated with the Lagrangian pair
\((\mc V,\mc V^\#)\), with respect to the domain basis \(\beta^\#\) and
the codomain basis \(\beta\).
\end{defn}

The general Weyl-function construction in this boundary-symplectic setting
was developed in \cite{LiouKrein}.

We now make the preceding constructions explicit for the Friedrichs
Lagrangian pair \((\mc F,\mc F^\#)\). Let
\(\lambda\in\rho(\Delta_{\mc F})\), and choose \(\nu\) so that
\(\lambda=\nu(\nu+1)\).

Define \(u_\lambda^0\in\mc D_{\max}(\Delta)\) by
\[
    u_\lambda^0(r,\theta)
    =
    \frac{1}{n\sqrt{2\pi}}
    \left[
        Q_\nu\left(\frac{1-r^{2n}}{1+r^{2n}}\right)
        -
        \frac{\pi}{2}\cot(\pi\nu)
        P_\nu\left(\frac{1-r^{2n}}{1+r^{2n}}\right)
    \right].
\]
Then \(u_\lambda^0\) is the unique formal solution of
\eqref{eq:model} satisfying
\(\widetilde{\pi}_{\mc F^\#}(u_\lambda^0)=f_0^\#\).
Consequently,
\(\Gamma_{\mc F,\mc F^\#}(\lambda)f_0^\#=u_\lambda^0\).

Since
\[
    \pi_S(u_\lambda^0)
    =
    -
    \frac{
        \gamma_E+\psi(\nu+1)
        +\frac{\pi}{2}\cot(\pi\nu)
    }{n}
    f_0
    -
    \frac{\pi}{2n\sin(\pi\nu)}g_0
    +
    f_0^\#,
\]
we obtain
\[
    M_{\mc F,\mc F^\#}(\lambda)f_0^\#
    =
    -
    \frac{
        \gamma_E+\psi(\nu+1)
        +\frac{\pi}{2}\cot(\pi\nu)
    }{n}
    f_0
    -
    \frac{\pi}{2n\sin(\pi\nu)}g_0.
\]

For \(k\in\mb Z\) with \(0<|k|\leq n-1\), define
\(u_\lambda^k\in\mc D_{\max}(\Delta)\) by
\[
    u_\lambda^k(r,\theta)
    =
    \frac{\Gamma(1-\alpha_k)}{\sqrt{4\pi|k|}}
    \left[
        P_\nu^{\alpha_k}\left(
            \frac{1-r^{2n}}{1+r^{2n}}
        \right)
        -
        \frac{\Gamma(\nu+\alpha_k+1)}
             {\Gamma(\nu-\alpha_k+1)}
        \frac{\sin(\pi\nu)}
             {\sin\bigl(\pi(\nu-\alpha_k)\bigr)}
        P_\nu^{-\alpha_k}\left(
            \frac{1-r^{2n}}{1+r^{2n}}
        \right)
    \right]
    e^{ik\theta}.
\]
Then \(u_\lambda^k\) is the unique formal solution satisfying
\(\widetilde{\pi}_{\mc F^\#}(u_\lambda^k)=f_k^\#\). Consequently,
\(\Gamma_{\mc F,\mc F^\#}(\lambda)f_k^\#=u_\lambda^k\).

Since
\[
\begin{aligned}
    \pi_S(u_\lambda^k)
    ={}&
    -
    \frac{\Gamma(1-\alpha_k)}{\Gamma(1+\alpha_k)}
    \frac{\Gamma(\nu+\alpha_k+1)}
         {\Gamma(\nu-\alpha_k+1)}
    \frac{\sin(\pi\nu)}
         {\sin\bigl(\pi(\nu-\alpha_k)\bigr)}
    f_k
    +
    f_k^\#
    \\[2mm]
    &\quad
    -
    \frac{\Gamma(1-\alpha_k)}{\Gamma(1+\alpha_k)}
    \frac{\Gamma(\nu+\alpha_k+1)}
         {\Gamma(\nu-\alpha_k+1)}
    \frac{\sin(\pi\alpha_k)}
         {\sin\bigl(\pi(\nu-\alpha_k)\bigr)}
    g_{-k},
\end{aligned}
\]
we obtain
\[
\begin{aligned}
    M_{\mc F,\mc F^\#}(\lambda)f_k^\#
    ={}&
    -
    \frac{\Gamma(1-\alpha_k)}{\Gamma(1+\alpha_k)}
    \frac{\Gamma(\nu+\alpha_k+1)}
         {\Gamma(\nu-\alpha_k+1)}
    \frac{\sin(\pi\nu)}
         {\sin\bigl(\pi(\nu-\alpha_k)\bigr)}
    f_k
    \\[2mm]
    &\quad
    -
    \frac{\Gamma(1-\alpha_k)}{\Gamma(1+\alpha_k)}
    \frac{\Gamma(\nu+\alpha_k+1)}
         {\Gamma(\nu-\alpha_k+1)}
    \frac{\sin(\pi\alpha_k)}
         {\sin\bigl(\pi(\nu-\alpha_k)\bigr)}
    g_{-k}.
\end{aligned}
\]
The expressions involving
\(\sin\bigl(\pi(\nu-\alpha_k)\bigr)\) are understood by meromorphic
continuation at removable singularities.

Define \(v_\lambda^0\in\mc D_{\max}(\Delta)\) by
\[
    v_\lambda^0(r,\theta)
    =
    \frac{1}{n\sqrt{2\pi}}
    \left[
        Q_\nu\left(\frac{r^{2n}-1}{1+r^{2n}}\right)
        -
        \frac{\pi}{2}\cot(\pi\nu)
        P_\nu\left(\frac{r^{2n}-1}{1+r^{2n}}\right)
    \right].
\]
Equivalently,
\[
    v_\lambda^0(r,\theta)
    =
    -
    \frac{\pi}{2n\sqrt{2\pi}\sin(\pi\nu)}
    P_\nu\left(\frac{1-r^{2n}}{1+r^{2n}}\right).
\]
Then \(v_\lambda^0\) is the unique formal solution satisfying
\(\widetilde{\pi}_{\mc F^\#}(v_\lambda^0)=g_0^\#\). Consequently,
\[
    \Gamma_{\mc F,\mc F^\#}(\lambda)g_0^\#
    =
    v_\lambda^0.
\]
Since
\[
    \pi_S(v_\lambda^0)
    =
    -
    \frac{\pi}{2n\sin(\pi\nu)}f_0
    -
    \frac{
        \gamma_E+\psi(\nu+1)
        +\frac{\pi}{2}\cot(\pi\nu)
    }{n}
    g_0
    +
    g_0^\#,
\]
we obtain
\[
    M_{\mc F,\mc F^\#}(\lambda)g_0^\#
    =
    -
    \frac{\pi}{2n\sin(\pi\nu)}f_0
    -
    \frac{
        \gamma_E+\psi(\nu+1)
        +\frac{\pi}{2}\cot(\pi\nu)
    }{n}
    g_0.
\]

For \(k\in\mb Z\) with \(0<|k|\leq n-1\), define
\(v_\lambda^k\in\mc D_{\max}(\Delta)\) by
\[
\begin{aligned}
    v_\lambda^k(r,\theta)
    ={}&
    \frac{\Gamma(1-\alpha_k)}{\sqrt{4\pi|k|}}
    \Bigg[
        P_\nu^{\alpha_k}\left(
            \frac{r^{2n}-1}{1+r^{2n}}
        \right)
        \\
        &\qquad
        -
        \frac{\Gamma(\nu+\alpha_k+1)}
             {\Gamma(\nu-\alpha_k+1)}
        \frac{\sin(\pi\nu)}
             {\sin\bigl(\pi(\nu-\alpha_k)\bigr)}
        P_\nu^{-\alpha_k}\left(
            \frac{r^{2n}-1}{1+r^{2n}}
        \right)
    \Bigg]e^{-ik\theta}.
\end{aligned}
\]
Then \(v_\lambda^k\) is the unique formal solution of
\eqref{eq:model} satisfying
\(\widetilde{\pi}_{\mc F^\#}(v_\lambda^k)=g_k^\#\). Consequently,
\[
    \Gamma_{\mc F,\mc F^\#}(\lambda)g_k^\#
    =
    v_\lambda^k.
\]
Since
\[
\begin{aligned}
    \pi_S(v_\lambda^k)
    ={}&
    -
    \frac{\Gamma(1-\alpha_k)}{\Gamma(1+\alpha_k)}
    \frac{\Gamma(\nu+\alpha_k+1)}
         {\Gamma(\nu-\alpha_k+1)}
    \frac{\sin(\pi\alpha_k)}
         {\sin\bigl(\pi(\nu-\alpha_k)\bigr)}
    f_{-k}
    \\[2mm]
    &\quad
    -
    \frac{\Gamma(1-\alpha_k)}{\Gamma(1+\alpha_k)}
    \frac{\Gamma(\nu+\alpha_k+1)}
         {\Gamma(\nu-\alpha_k+1)}
    \frac{\sin(\pi\nu)}
         {\sin\bigl(\pi(\nu-\alpha_k)\bigr)}
    g_k
    +
    g_k^\#,
\end{aligned}
\]
we obtain
\[
\begin{aligned}
    M_{\mc F,\mc F^\#}(\lambda)g_k^\#
    ={}&
    -
    \frac{\Gamma(1-\alpha_k)}{\Gamma(1+\alpha_k)}
    \frac{\Gamma(\nu+\alpha_k+1)}
         {\Gamma(\nu-\alpha_k+1)}
    \frac{\sin(\pi\alpha_k)}
         {\sin\bigl(\pi(\nu-\alpha_k)\bigr)}
    f_{-k}
    \\[2mm]
    &\quad
    -
    \frac{\Gamma(1-\alpha_k)}{\Gamma(1+\alpha_k)}
    \frac{\Gamma(\nu+\alpha_k+1)}
         {\Gamma(\nu-\alpha_k+1)}
    \frac{\sin(\pi\nu)}
         {\sin\bigl(\pi(\nu-\alpha_k)\bigr)}
    g_k.
\end{aligned}
\]

We remark that
\(
    \bigl\{
        u_\lambda^k,v_\lambda^k:
        |k|\leq n-1
    \bigr\}
\)
forms a basis of \(\ker(\Delta-\lambda)\). Indeed, these formal solutions
are the images under \(\Gamma_{\mc F,\mc F^\#}(\lambda)\) of the basis
\(
    \bigl\{
        f_k^\#,g_k^\#:
        |k|\leq n-1
    \bigr\}
\)
of \(\mc F^\#\). This also verifies directly that
\(\dim_{\mb C}\ker(\Delta-\lambda)=4n-2=\dim_{\mb C}\mc F^\#\).

To exhibit the block structure of the \(S\)-matrix, we group the basis
elements according to the pairs of Fourier modes coupled by the coordinate
transformation \(y=x^{-1}\). Let
\(\beta_0^\#=(f_0^\#,g_0^\#)\) and \(\beta_0=(f_0,g_0)\). For
\(1\leq k\leq n-1\), set
\(\beta_k^\#=(f_k^\#,g_{-k}^\#)\),
\(\beta_k=(f_k,g_{-k})\),
\(\beta_{-k}^\#=(f_{-k}^\#,g_k^\#)\), and
\(\beta_{-k}=(f_{-k},g_k)\). We form the ordered bases
\(\beta^\#=(\beta_0^\#;\beta_1^\#,\beta_{-1}^\#;\ldots;
\beta_{n-1}^\#,\beta_{-(n-1)}^\#)\) of \(\mc F^\#\) and
\(\beta=(\beta_0;\beta_1,\beta_{-1};\ldots;
\beta_{n-1},\beta_{-(n-1)})\) of \(\mc F\).

With respect to the ordered basis pair
\((\beta_0^\#,\beta_0)\), the zero-mode block is
\[
    S_0(\lambda)
    =
    -
    \frac{1}{n}
    \begin{pmatrix}
        \gamma_E+\psi(\nu+1)+\dfrac{\pi}{2}\cot(\pi\nu)
        &
        \dfrac{\pi}{2\sin(\pi\nu)}
        \\[3mm]
        \dfrac{\pi}{2\sin(\pi\nu)}
        &
        \gamma_E+\psi(\nu+1)+\dfrac{\pi}{2}\cot(\pi\nu)
    \end{pmatrix}.
\]
For \(1\leq k\leq n-1\), the restrictions of
\(M_{\mc F,\mc F^\#}(\lambda)\) to
\(\operatorname{span}_{\mb C}\{f_k^\#,g_{-k}^\#\}\) and
\(\operatorname{span}_{\mb C}\{f_{-k}^\#,g_k^\#\}\) have the same
matrix representation. Namely,
\[
\begin{aligned}
    S_k(\lambda)
    &:=
    \left[
        M_{\mc F,\mc F^\#}(\lambda)
        \big|_{\operatorname{span}_{\mb C}\{f_k^\#,g_{-k}^\#\}}
    \right]_{\beta_k^\#}^{\beta_k}
    \\
    &=
    \left[
        M_{\mc F,\mc F^\#}(\lambda)
        \big|_{\operatorname{span}_{\mb C}\{f_{-k}^\#,g_k^\#\}}
    \right]_{\beta_{-k}^\#}^{\beta_{-k}}
    \\
    &=
    -
    \frac{\Gamma(1-\alpha_k)}{\Gamma(1+\alpha_k)}
    \frac{\Gamma(\nu+\alpha_k+1)}
         {\Gamma(\nu-\alpha_k+1)}
    \frac{1}{\sin\bigl(\pi(\nu-\alpha_k)\bigr)}
    \begin{pmatrix}
        \sin(\pi\nu) & \sin(\pi\alpha_k)
        \\[2mm]
        \sin(\pi\alpha_k) & \sin(\pi\nu)
    \end{pmatrix}.
\end{aligned}
\]
Consequently, relative to the ordered bases
\((\beta^\#,\beta)\), the \(S\)-matrix is block diagonal:
\[
    S_{\mc F,\mc F^\#}^{\beta,\beta^\#}(\lambda)
    =
    \begin{pmatrix}
        S_0(\lambda) & 0 & 0 & \cdots & 0 & 0 \\
        0 & S_1(\lambda) & 0 & \cdots & 0 & 0 \\
        0 & 0 & S_1(\lambda) & \cdots & 0 & 0 \\
        \vdots & \vdots & \vdots & \ddots & \vdots & \vdots \\
        0 & 0 & 0 & \cdots & S_{n-1}(\lambda) & 0 \\
        0 & 0 & 0 & \cdots & 0 & S_{n-1}(\lambda)
    \end{pmatrix}.
\]

The explicit formulas above verify, in the present model, the general
relation between the poles of the Weyl function and the spectrum of the
reference extension established in \cite{LiouKrein}. Indeed, up to the
symmetry \(\nu\mapsto-\nu-1\), the zero-mode block has poles at
\(\nu=q\), \(q\in\mb Z_{\geq0}\), whereas the nonzero-mode blocks have
poles at
\[
    \nu=\alpha_k+q,
    \qquad
    1\leq k\leq n-1,\quad q\in\mb Z_{\geq0}.
\]
Here the displayed expressions are understood meromorphically, so that
apparent singularities which are removable are not counted as poles. Under
the relation \(\lambda=\nu(\nu+1)\), these poles reproduce the Friedrichs
spectrum of Theorem~\ref{thm:friedrichs-spectrum} as a set.
Indeed, if \(|m|=\ell n+k\), where
\(0\leq k\leq n-1\), then
\(\alpha_m=\ell+\alpha_k\), and hence
\[
    \lambda_{q,m}
    =
    \lambda_{q+\ell,k}.
\]
Thus the eigenvalues arising from the modes \(|m|\geq n\) introduce no
additional spectral values, although they do contribute to the
multiplicities.

We now derive the relation between the \(S\)-matrix and the Legendre
connection matrix. To this end, we introduce some notation. Let
\(T:\mc F^\#\to\mc F\) be a linear map. Denote by
\(\iota_0^\#:\mc F_0^\#\to\mc F^\#\) and
\(\iota_\infty^\#:\mc F_\infty^\#\to\mc F^\#\) the canonical
inclusions, and by \(p_0:\mc F\to\mc F_0\) and
\(p_\infty:\mc F\to\mc F_\infty\) the canonical projections.

Define
\(T_{11}=p_0\circ T\circ\iota_0^\#\),
\(T_{12}=p_0\circ T\circ\iota_\infty^\#\),
\(T_{21}=p_\infty\circ T\circ\iota_0^\#\), and
\(T_{22}=p_\infty\circ T\circ\iota_\infty^\#\). Thus
\[
    T
    =
    \begin{pmatrix}
        T_{11} & T_{12} \\
        T_{21} & T_{22}
    \end{pmatrix}
\]
with respect to the decompositions
\(\mc F^\#=\mc F_0^\#\oplus\mc F_\infty^\#\) and
\(\mc F=\mc F_0\oplus\mc F_\infty\).

Let \(h^\#\in\mc F^\#\), and write
\(h^\#=h_0^\#+h_\infty^\#\), where
\(h_0^\#\in\mc F_0^\#\) and
\(h_\infty^\#\in\mc F_\infty^\#\). Set
\[
    u_\lambda
    =
    \Gamma_{\mc F,\mc F^\#}(\lambda)h^\#.
\]
Then
\[
    \pi_S(u_\lambda)
    =
    M_{\mc F,\mc F^\#}(\lambda)h^\#
    +
    h^\#.
\]
Write
\[
    M_{\mc F,\mc F^\#}(\lambda)
    =
    \begin{pmatrix}
        M_{11}(\lambda) & M_{12}(\lambda) \\
        M_{21}(\lambda) & M_{22}(\lambda)
    \end{pmatrix}
\]
with respect to the decompositions
\(\mc F^\#=\mc F_0^\#\oplus\mc F_\infty^\#\) and
\(\mc F=\mc F_0\oplus\mc F_\infty\). It follows that
\[
\begin{aligned}
    \pi_0(u_\lambda)
    &=
    M_{11}(\lambda)h_0^\#
    +
    M_{12}(\lambda)h_\infty^\#
    +
    h_0^\#,
    \\
    \pi_\infty(u_\lambda)
    &=
    M_{21}(\lambda)h_0^\#
    +
    M_{22}(\lambda)h_\infty^\#
    +
    h_\infty^\#.
\end{aligned}
\]

Similarly, write the boundary connection map as
\[
    \mc C_\nu
    =
    \begin{pmatrix}
        C_{11}(\nu) & C_{12}(\nu) \\
        C_{21}(\nu) & C_{22}(\nu)
    \end{pmatrix}
\]
with respect to the decompositions
\(\mc V_0=\mc F_0\oplus\mc F_0^\#\) and
\(\mc V_\infty=\mc F_\infty\oplus\mc F_\infty^\#\). Since
\(\pi_\infty(u_\lambda)=\mc C_\nu\pi_0(u_\lambda)\), we obtain
\[
\begin{aligned}
    M_{21}(\lambda)h_0^\#
    +
    M_{22}(\lambda)h_\infty^\#
    &=
    C_{11}(\nu)
    \bigl(
        M_{11}(\lambda)h_0^\#
        +
        M_{12}(\lambda)h_\infty^\#
    \bigr)
    +
    C_{12}(\nu)h_0^\#,
    \\
    h_\infty^\#
    &=
    C_{21}(\nu)
    \bigl(
        M_{11}(\lambda)h_0^\#
        +
        M_{12}(\lambda)h_\infty^\#
    \bigr)
    +
    C_{22}(\nu)h_0^\#.
\end{aligned}
\]
Since \(h_0^\#\) and \(h_\infty^\#\) are arbitrary, comparison of
coefficients gives
\[
    C_{21}(\nu)M_{11}(\lambda)+C_{22}(\nu)=0,
    \qquad
    C_{21}(\nu)M_{12}(\lambda)=I_{\mc F_\infty^\#},
\]
and
\[
    M_{21}(\lambda)
    =
    C_{11}(\nu)M_{11}(\lambda)+C_{12}(\nu),
    \qquad
    M_{22}(\lambda)
    =
    C_{11}(\nu)M_{12}(\lambda).
\]

Since
\(C_{21}(\nu)M_{12}(\lambda)=I_{\mc F_\infty^\#}\), the map
\(C_{21}(\nu):\mc F_0\to\mc F_\infty^\#\) is surjective. Since
\(\mc F_0\) and \(\mc F_\infty^\#\) have the same finite dimension,
\(C_{21}(\nu)\) is invertible. Hence
\(M_{12}(\lambda)=C_{21}(\nu)^{-1}\). Consequently,
\[
    M_{\mc F,\mc F^\#}(\lambda)
    =
    \begin{pmatrix}
        -C_{21}(\nu)^{-1}C_{22}(\nu)
        &
        C_{21}(\nu)^{-1}
        \\[2mm]
        C_{12}(\nu)
        -
        C_{11}(\nu)C_{21}(\nu)^{-1}C_{22}(\nu)
        &
        C_{11}(\nu)C_{21}(\nu)^{-1}
    \end{pmatrix}.
\]
This leads to the following theorem.

\begin{thm}\label{thm:weyl-connection-matrix}
Let \(\lambda\in\rho(\Delta_{\mc F})\), and choose \(\nu\in\mb C\)
such that \(\lambda=\nu(\nu+1)\). Write the boundary connection map
\(\mc C_\nu:\mc V_0\to\mc V_\infty\), relative to the decompositions
\(\mc V_0=\mc F_0\oplus\mc F_0^\#\) and
\(\mc V_\infty=\mc F_\infty\oplus\mc F_\infty^\#\), as
\[
    \mc C_\nu
    =
    \begin{pmatrix}
        C_{11}(\nu) & C_{12}(\nu) \\
        C_{21}(\nu) & C_{22}(\nu)
    \end{pmatrix}.
\]
Then \(C_{21}(\nu):\mc F_0\to\mc F_\infty^\#\) is invertible, and the
Weyl function \(M_{\mc F,\mc F^\#}(\lambda):\mc F^\#\to\mc F\) is
given, relative to the decompositions
\(\mc F^\#=\mc F_0^\#\oplus\mc F_\infty^\#\) and
\(\mc F=\mc F_0\oplus\mc F_\infty\), by
\[
    M_{\mc F,\mc F^\#}(\lambda)
    =
    \begin{pmatrix}
        -C_{21}(\nu)^{-1}C_{22}(\nu)
        &
        C_{21}(\nu)^{-1}
        \\[2mm]
        C_{12}(\nu)
        -
        C_{11}(\nu)C_{21}(\nu)^{-1}C_{22}(\nu)
        &
        C_{11}(\nu)C_{21}(\nu)^{-1}
    \end{pmatrix}.
\]
Consequently, taking matrix representations with respect to the chosen
symplectically dual bases, the \(S\)-matrix is determined explicitly by
the Legendre connection matrix through the boundary connection map
\[
    \mc C_\nu
    =
    \mc A_\infty^\nu
    \circ
    \mc M_\nu
    \circ
    (\mc A_0^\nu)^{-1}.
\]
\end{thm}

This theorem explains how the Weyl function can be recovered from the
boundary connection map. The map \(\mc C_\nu\) relates the boundary data
of the same formal solution at the two conic points: once its boundary
data at \(0\) are known, \(\mc C_\nu\) determines the corresponding
boundary data at \(\infty\). The Weyl function, on the other hand, starts
with prescribed singular boundary data in \(\mc F^\#\) and determines
the corresponding regular boundary data in \(\mc F\). The formula in
the theorem is obtained by combining these two descriptions through the
gluing relation
\(\pi_\infty(u_\lambda)=\mc C_\nu\pi_0(u_\lambda)\) and solving for the
regular boundary components.

Since \(\mc C_\nu\) is induced by the Legendre connection matrices, the
theorem shows that the classical connection formulas for Legendre
functions determine the Weyl function, and hence the \(S\)-matrix, of
the conic Laplacian. Thus the transition between local Legendre
solutions at \(0\) and \(\infty\) determines the global relation between
the singular and regular boundary data of formal solutions on
\(\mb P^1\).

The block \(C_{21}(\nu):\mc F_0\to\mc F_\infty^\#\) plays a particularly
important role. It records how regular boundary data at \(0\) contribute
to singular boundary data at \(\infty\). Its invertibility makes it
possible to recover uniquely the regular boundary data at \(0\) from
the singular boundary data at \(\infty\). Moreover, a vector
\(h_0\in\ker C_{21}(\nu)\) would give rise to a formal solution whose
boundary data are regular at both conic points, and hence to a
Friedrichs eigenfunction with eigenvalue \(\lambda\). Therefore the
invertibility of \(C_{21}(\nu)\) is the boundary-data manifestation of
the assumption \(\lambda\in\rho(\Delta_{\mc F})\).

\section{Conclusion}

For the projective conic model on \(\mb P^1\), we have shown that the
local Legendre connection problem and the global boundary-spectral
problem are two manifestations of the same underlying structure. The
boundary connection map induced by the classical Legendre connection
formulas determines the Weyl function and the \(S\)-matrix, while the
invertibility of its \(C_{21}\)-block characterizes the Friedrichs
resolvent condition. This identification gives a geometric
interpretation of the Legendre connection matrices and provides an
explicit bridge between special-function theory, boundary symplectic
geometry, and the spectral theory of conic Laplacians.

\section*{Acknowledgments}
The author is deeply grateful to the Lord Jesus Christ and to his family for sustaining him through a difficult period and for their steadfast encouragement and support.


\begin{thebibliography}{99}\large

\bibitem{Cheeger}
J.~Cheeger,
Spectral geometry of singular Riemannian spaces,
\emph{J. Differential Geom.} \textbf{18} (1983), no.~4, 575--657.

\bibitem{GilMendoza}
J.~B.~Gil and G.~A.~Mendoza,
Adjoints of elliptic cone operators,
\emph{Amer. J. Math.} \textbf{125} (2003), no.~2, 357--408.

\bibitem{GilKrainerMendoza}
J.~B.~Gil, T.~Krainer, and G.~A.~Mendoza,
Geometry and spectra of closed extensions of elliptic cone operators,
\emph{Canad. J. Math.} \textbf{59} (2007), no.~4, 742--794.

\bibitem{GilKrainerMendozaResolvents}
J.~B.~Gil, T.~Krainer, and G.~A.~Mendoza,
Resolvents of elliptic cone operators,
\emph{J. Funct. Anal.} \textbf{241} (2006), no.~1, 1--55.

\bibitem{HillairetKokotov}
L.~Hillairet and A.~Kokotov,
Krein formula and \(S\)-matrix for Euclidean surfaces with conical
singularities,
\emph{J. Geom. Anal.} \textbf{23} (2013), no.~3, 1498--1529.

\bibitem{LiouKrein}
J.-M.~Liou,
Krein's formula for conic Laplacians on compact Riemann surfaces,
arXiv preprint arXiv:2606.20818v1 (2026).

\bibitem{DLMF}
NIST Digital Library of Mathematical Functions,
Chapter~14: Legendre and related functions,
Version~1.2.7, release date June~15, 2026,
National Institute of Standards and Technology,
\texttt{https://dlmf.nist.gov/14}.

\bibitem{Troyanov}
M.~Troyanov,
Prescribing curvature on compact surfaces with conical singularities,
\emph{Trans. Amer. Math. Soc.} \textbf{324} (1991), no.~2, 793--821.

\end{thebibliography}
\end{document}